\documentclass{amsart}
\usepackage{latexsym,amsxtra,amscd,ifthen}
\usepackage{amsfonts}
\usepackage{verbatim}
\usepackage{amsmath}
\usepackage{amsthm}
\usepackage{amssymb}
\usepackage{url}
\usepackage{enumerate}
\usepackage{mathrsfs}
\usepackage{color}
\usepackage{array}
\usepackage{arydshln}
\usepackage{multirow}

\usepackage{geometry}
\usepackage{slashbox}
\usepackage{graphicx}
\usepackage{tikz}

\theoremstyle{plain}

\newtheorem{theorem}{Theorem}
\newtheorem{lemma}[theorem]{Lemma}
\newtheorem{proposition}[theorem]{Proposition}

\newtheorem{conjecture}[theorem]{Conjecture}

\numberwithin{theorem}{section}
\numberwithin{equation}{theorem}

\theoremstyle{definition}

\newtheorem*{question*}{Question}

\def\supprimes#1{\ifcase\value{#1}
  \relax \or $'$ \or $''$ \or $'''$ \or $''''$
  \else out of range\fi}

\begin{document}

\title[Geck's Conjecture and GGGR in Bad Characteristic]
{Geck's Conjecture and the Generalized Gelfand-Graev Representations in Bad Characteristic}

\author{Junbin Dong}
\address{Institute of Mathematical Sciences, ShanghaiTech University, 393 Middle Huaxia Road, Pudong, Shanghai 201210, PR China.}
\email{dongjunbin1990@126.com}

\author{Gao Yang}
\address{Department of Mathematics, Harbin Engineering University, 145 Nantong Street, Harbin 150001, PR China.}
\email{yanggao\_670206@163.com}






\maketitle

\begin{abstract}

For a connected reductive algebraic group $G$ defined over a finite field $\mathbb F_q$, Kawanaka introduced the generalized Gelfand-Graev representations (GGGRs for short) of the finite group $G(\mathbb F_q)$ in the case where $q$ is a power of a good prime for $G$. This representation has been widely studied and used in various contexts. Recently, Geck proposed a conjecture, characterizing Lusztig's special unipotent classes in terms of weighted Dynkin diagrams. Based on this conjecture, he gave a guideline for extending the definition of GGGRs to the case where $q$ is a power of a bad prime for $G$. Here, we will give a proof of Geck's conjecture. Combined with Geck's pioneer work, our proof
verifies Geck's conjectural characterization of special unipotent classes, and completes his definition of GGGRs in bad characteristics.
\end{abstract}


\setcounter{section}{-1}
\section{Introduction} \label{s_intro}

\subsection{}
Let $p$ be a prime and $\mathbf k = \overline{\mathbb F}_p$ be an algebraic closure of the field with $p$ elements.
 Let $G$ be a connected reductive algebraic group over $\mathbf k$, which is defined over the finite field $\mathbb F_q \subseteq \mathbf k$, where $q$ is a power of $p$.
 Let $F: G \rightarrow G$ be the corresponding Frobenius map, and $G^F =\{ g \in G \mid F(g) = g \}$.
 Fix an $F$-stable maximal torus $T$ and an $F$-stable Borel subgroup $B$ containing $T$. Let $\Phi$ be the root system of $G$ with respect to $T$, and $\Pi \subseteq \Phi$ be the set of simple roots determined by $B$. Let $\Phi^+$ (resp. $\Phi^-$) be the set of positive (resp. negative) roots with respect to $\Pi$.
  For each $\alpha \in \Phi$, there is an homomorphism $x_{\alpha} : \mathbf k^+ \rightarrow G$, $u \mapsto x_{\alpha}(u)$, which is an isomorphism onto its image, and satisfies $tx_{\alpha}(u)t^{-1} = x_{\alpha}(\alpha(t)u)$ for all $t \in T$ and $u \in \mathbf k$.
  Set $U_{\alpha} = \{ x_{\alpha}(u) \mid u \in \mathbf k^+\}$. Then $G = \langle T, U_{\alpha} ( \alpha \in \Phi) \rangle$.

 Let $\mathfrak g$ be the Lie algebra of $G$. Let $d_0x_{\alpha}: \mathfrak g \rightarrow \mathfrak g$ be the differential of $x_{\alpha}$, and $\mathfrak g_{\alpha} = d_0x_{\alpha}(\mathfrak g)$.
 Then there is a direct sum decomposition of $\mathfrak g$:
\begin{eqnarray}
\mathfrak g = \mathfrak t \oplus \bigoplus_{\alpha \in \Phi} \mathfrak g_{\alpha}, \nonumber
\end{eqnarray}
where $\mathfrak t=\text{Lie}(T)$ is the Cartan subalgebra of $\mathfrak g$.

\subsection{}
 When $p$ is a good prime for $G$, Kawanaka defined a representation $\Gamma_C$ of $G^F$, associated to each unipotent $G^F$-conjugacy class $C$ in $G^F$, see \cite{K1,K2,K3}. If $C$ consists of regular unipotent elements, then $\Gamma_C$ is a Gelfand-Graev representation, as defined in \cite{St}.
 Thus  $\Gamma_C$ is called a \emph{generalized Gelfand-Graev representation} (GGGR for short) for an arbitrary conjugacy class $C$.

 The original purpose for Kawanaka to define GGGRs was to prove Ennola's conjecture. Far beyond this proof, GGGRs have been applied to many other areas of representation theory of finite groups of Lie type, see a survey given by Geck in \cite[\S 1]{G}.

 When $p$ is a bad prime for $G$, one can not expect a good definitions of GGGRs for all unipotent elements in $G^F$ according to some easy examples. However it seems reasonable to restrict oneself to those unipotent classes which "come form characteristic $0$". In the paper \cite{G}, Geck proposed a guideline to a good definition of GGGRs which we will explain briefly in the following.

\subsection{}
Let $\Delta$ be the set of weighted Dynkin diagrams. Any Dynkin diagram $d \in \Delta$ is a map $d: \Phi \rightarrow \mathbb Z $ satisfying:

\noindent (a) $d(-\alpha)=-d(\alpha)$ for all $\alpha\in \Phi$ and $d(\alpha+\beta)=d(\alpha)+d(\beta)$ for all $\alpha,\beta\in \Phi$ such that $\alpha+\beta \in \Phi$.

\noindent (b) $d(\alpha)\in \{0,1,2\}$ for every simple root $\alpha\in \Pi$.

Let $d \in \Delta$.
For $i\in \mathbb Z$, we set $\Phi_{d,i}:=\{\alpha\in \Phi \mid d(\alpha)=i\}$ and define
\begin{eqnarray}
\mathfrak g_{d}(i): =
\begin{cases}
\bigoplus_{\alpha\in \Phi_{d,i}} \mathfrak g_{\alpha} & \text{ if } i\ne 0, \\

\mathfrak t \oplus  \bigoplus_{\alpha\in \Phi_{d,0}} \mathfrak g_{\alpha}    & \text{ if } i= 0.
\end{cases} \nonumber
\end{eqnarray}
Furthermore we define $U_{d,i}  = \langle U_{\alpha} \mid \alpha \in \Phi_{d,j} \text{ for all } j \geq i \rangle$. Then put $P_{d} = \langle T, U_{\alpha} \mid \alpha \in \Phi_{d,i} \text{ for all } i \geq 0 \rangle$, which is a parabolic subgroup of $G$.

\smallskip

The coadjoint action of $G$ on the dual vector space $\mathfrak g^*$ of $\mathfrak g$ is denoted as $g \cdot \xi$, and defined by $(g \cdot \xi)(y ) = \xi(\operatorname{Ad}(g^{-1})(y))$ for $g \in G$, $\xi \in \mathfrak g^*$ and $y \in \mathfrak g$.
Let $\dagger: \mathfrak g \rightarrow \mathfrak g$ be an opposition $\mathbb F_q$-automorphism, as defined in \cite{K0}. It will induce a map $\dagger: \mathfrak g^* \rightarrow \mathfrak g^*$, defined by $\xi^{\dagger}(y) = \xi(y^{\dagger})$ for $\xi \in \mathfrak g^*$ and $y \in \mathfrak g$. Let $G_{\xi} = \{ g \in G \mid g \cdot \xi = \xi \}$ for $\xi \in \mathfrak g^*$.

\smallskip
For $d \in \Delta$ and given a homomorphism  $\lambda \in \operatorname{Hom}_{\mathbf k }(\mathfrak g_{d}(2), \mathbf k )$, we obtain an alternating form $\sigma_{\lambda }: \mathfrak g_{d}(1)\times \mathfrak g_{d}(1) \rightarrow \mathbf k $.  Following \cite{G}, we call $\lambda \in \operatorname{Hom}_{\mathbf k}(\mathfrak g_d(2), \mathbf k)$ in \emph{sufficiently general position}, if the following conditions (K1) and (K2) hold.

\noindent (K1) $G_{\lambda^{\dagger}} \subseteq P_d$, where $\lambda$ is regarded as element of $\mathfrak g^*$ whose restriction on $\mathfrak g_d(i)$ is zero for all $i \neq 2$.

\noindent (K2) If $\mathfrak g_d(1) \neq \{0\}$, then the radical of $\sigma_{\lambda}: \mathfrak g_d(1) \times \mathfrak g_d(1) \rightarrow \mathbf k$ is zero.

%
%
%

For $d \in \Delta$, if there exists $\lambda \in \operatorname{Hom}_{\mathbf k}(\mathfrak g_d(2), \mathbf k)$ in sufficiently general position, then the GGGR $\Gamma_{d,\lambda}$ can be defined through the process described in \cite[\S 2.8]{G}.
Thus the existence of $\lambda$ in sufficiently general position becomes the key point in defining a GGGR.

By the arguments in \cite[Remarks 3.6-3.7]{G}, we have the following properties:

\noindent (A1) the set of $\lambda \in \operatorname{Hom}_{\mathbf k}(\mathfrak g_d(2), \mathbf k)$ satisfying (K1) is open dense in $\operatorname{Hom}_{\mathbf k}(\mathfrak g_d(2), \mathbf k)$;

\noindent  (A2) the set of $\lambda \in \operatorname{Hom}_{\mathbf k}(\mathfrak g_d(2), \mathbf k)$ satisfying (K2) is open in $\operatorname{Hom}_{\mathbf k}(\mathfrak g_d(2), \mathbf k)$.

\noindent Consequently,

\noindent  (A3) there exists $\lambda \in \operatorname{Hom}_{\mathbf k}(\mathfrak g_d(2), \mathbf k)$ in sufficiently general position if and only if there exists $\lambda \in \operatorname{Hom}_{\mathbf k}(\mathfrak g_d(2), \mathbf k)$ satisfying (K2).

On the existence of $\lambda \in \operatorname{Hom}_{\mathbf k}(\mathfrak g_d(2), \mathbf k)$ satisfying (K2), Geck proposed several conjectures which we will prove in this paper.

\subsection{}
Let $\mathfrak C$ be the set of unipotent classes coming from characteristic $0$, as defined in \cite[\S 3.1]{G}. Then $\mathfrak C$ can be identified with $\Delta$. Let $\mathfrak C_{\text{spec}}$ be the set of special unipotent orbits, and $\Delta_{\text{spec}}$ be the subset of $\Delta$ which is identified with $\mathfrak C_{\text{spec}}$.

For any positive integer $n$, there exists a canonical map
$$\Psi_n: \operatorname{Irr}(G^{F^n}) \rightarrow \{ F^n-\text{stable unipotent classes} \},$$
sending an irreducible representation of $G^{F^n}$ to its unipotent support, see \cite[\S 4.1]{G}. By \cite[Remark 3.9]{GM}, the image of this map is contained in $\mathfrak C$. So the map $\Psi_n$ induces a map
$$\widetilde{\Psi}_n: \operatorname{Irr}(G^{F^n}) \rightarrow \Delta.$$
Let $\Delta_{\mathbf k} =  \operatorname{\bigcup}\limits_{n \in \mathbb Z_+} \operatorname{Im} \widetilde{\Psi}_n$.

\setcounter{theorem}{4}
\begin{conjecture}\cite[Conjecture 4.9]{G} \label{c_4.9}
Let $d \in \Delta$. Then $d \in \Delta_{\mathbf k}$ if and only if either  $\mathfrak g_{d}(1)=\{0\}$, or there exists $\lambda \in \operatorname{Hom}_{\mathbf k}(\mathfrak g_d(2), \mathbf k)$ satisfying (K2).
\end{conjecture}

When $p$ is good for $G$, we have $\Delta =\Delta_{\mathbf k}$, and there always exists $\lambda \in \operatorname{Hom}_{\mathbf k}(\mathfrak g_d(2), \mathbf k)$ in sufficiently general position (see \cite[Remarks 3.5, 4.2(b)]{G}). Thus Conjecture \ref{c_4.9} holds. By \cite[Corollary 5.11]{G}, Conjecture \ref{c_4.9} holds for $G$ of exceptional type $G_2$, $F_4$, $E_6$, $E_7$ or $E_8$.
Now we are going to prove Conjecture \ref{c_4.9} for $G$ of classical type $B_n$, $C_n$ or $D_n$ and $p$ being bad for $G$.

\setcounter{subsection}{5}
\subsection{}

Geck also formulated another conjecture \cite[Conjecture 4.10]{G}  to determine the the special weighted Dynkin diagram which can be regarded as an integral version of Conjecture \ref{c_4.9}.

For $d \in \Delta$, put
%
\begin{eqnarray}
 \mathfrak g_{\mathbb Z,d}(i) = \langle e_{\alpha} \mid d(\alpha) = i \rangle_{\mathbb Z}. \nonumber
\end{eqnarray}

\noindent Given a homomorphism  $\lambda \in \operatorname{Hom}_{\mathbb Z}(\mathfrak g_{\mathbb Z,d}(2), \mathbb Z)$, we obtain an alternating form $\sigma_{\lambda }: \mathfrak g_{\mathbb Z,d}(1)\times \mathfrak g_{\mathbb Z,d}(1) \rightarrow \mathbb Z$ and then we may consider its Gram matrix
$$\mathscr G_{d,\lambda } = (\lambda ([e_{\alpha},e_{\beta}]))_{\alpha,\beta \in \Phi_{d,1}}.$$
If this matrix $\mathscr G_{d,\lambda }$ has determinant $\pm 1$, then we say that $\sigma_{\lambda }$ is  \emph{non-degenerate}  over $\mathbb Z$. We often say $\lambda $ is non-degenerate if there is no confusion.

\setcounter{theorem}{6}
\begin{conjecture}\cite[Conjecture 4.10]{G} \label{c_4.10}
Let $d \in \Delta$. Then $d \in \Delta_{\text{spec}}$ if and only if either $\mathfrak g_{\mathbb Z,d}(1)=\{0\}$, or there exists a homomorphism $\lambda : \mathfrak g_{\mathbb Z,d}(2) \rightarrow \mathbb Z$ such that $\sigma_{\lambda }$ is non-degenerate over $\mathbb Z$.
\end{conjecture}

 As explained before, we only have to prove Conjecture \ref{c_4.9} under the assumption that $G$ is of classical type $B_n$, $C_n$ or $D_n$, and $p$ is bad for $G$. Under this assumption, according to \cite[Proposition 4.3]{G}, we have  $\Delta_{\mathbf k} = \Delta_{\text{spec}}$.
We formulate the following theorem (Theorem \ref{mainThm1}) which implies Conjecture  \ref{c_4.9} and Conjecture \ref{c_4.10}  for $G$ of classical type $A_n$, $B_n$, $C_n$ or $D_n$. Note that type $A_n$ for Conjecture \ref{c_4.9} is already valid, but it is of  independent interest for  Conjecture \ref{c_4.10}.





\begin{theorem} \label{mainThm1}
Assume that $G$ is of classical type $A_n$, $B_n$, $C_n$ or $D_n$. Let $d \in \Delta$.
\begin{enumerate}
\item If $d \in \Delta_{\text{spec}}$, then $\mathfrak g_{\mathbb Z,d}(1) =\{0\}$, or there exists a non-degenerate $\lambda \in \operatorname{Hom}_{\mathbb Z}( \mathfrak g_{\mathbb Z,d}(2),\mathbb Z)$.
\item If $d \not\in \Delta_{\text{spec}}$ and $\operatorname{char} \mathbf k =2$, then $\det \mathscr G_{d,\lambda} = 0$ for any $\lambda \in \operatorname{Hom}_{\mathbf k}(\mathfrak g_d(2), \mathbf k)$.
\end{enumerate}
\end{theorem}

With Conjecture \ref{c_4.9} proved, we complete Geck's definition of GGGR in bad characteristics.

\setcounter{subsection}{8}
\subsection{}

This paper is organized as follows. Section 1 contains some preliminaries, 
as well as some reductions of Geck's conjecture. Section 2 deals with the determinant of certain symmetric matrices over a field with even characteristic.
In Section 3 we give the proof of Geck's conjecture in the type $A$ case.
To deal with the type $B$, $C$, $D$ cases, we introduce the notion of faithful maps and show that the proof can be reduced to considering the faithful maps.
Then Sections 5 to 8 are devoted to the proof of Geck's conjecture through case by case discussions. In the last section we give some remarks.

\medskip

\noindent Notations:
\noindent (1) Let $a ,b \in \mathbb Z$. We write $a \equiv b$ if $a - b \in 2\mathbb Z$.

\noindent (2) If $a \leq b$, then we denote the set $\{ i \in \mathbb Z \mid a \leq i \leq b \}$ as $[a,b]$.

\noindent (3) Write $\beta \in \Phi$ as $\beta = \sum_{\alpha \in \Pi} k_{\alpha} \alpha$ for $k_{\alpha} \in \mathbb Z$, then $k_{\alpha}$ is the \emph{multiplicity} of $\alpha \in \Pi$ in $\beta$, denoted as $[\beta: \alpha]$.

\noindent (4) Let $\mathscr P_n$ be the set of partitions of $n$, for any positive integer $n$.

\noindent (5) For $r\in \mathbb{R}$, we denote $ \lfloor r \rfloor $ the maximal integer less than $r$.

\section{Preliminaries}

\subsection{Root system} We always follow the notation and definitions given in \cite{G}.
As defined in Section \ref{s_intro}, there is a direct sum decomposition of $\mathfrak g$:
\begin{eqnarray}
\mathfrak g = \mathfrak t \oplus \bigoplus_{\alpha \in \Phi} \mathfrak g_{\alpha}, \nonumber
\end{eqnarray}
where $\mathfrak t=\text{Lie}(T)$ is the Cartan subalgebra of $\mathfrak g$. 
For each simple Lie algebra we can fix a choice of simple roots $\Pi$ and positive roots $\Phi^+ $. In this paper we consider the simple Lie algebra of classical types 
$A_n$, $B_n$, $C_n$ and $D_n$. So we list their root systems here for convenience.

\smallskip
\noindent (1) The fundamental system of type $A_n$ is

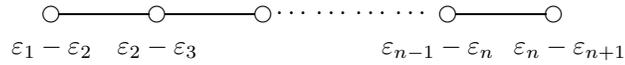
\begin{picture}(40,10)(-45,10)
\put(50,0){\circle{6}}
\put(90,0){\circle{6}}
\put(130,0){\circle{6}}
\put(135,0){$\dots$}
\put(150,0){$\dots$}
\put(165,0){$\dots$}
\put(180,0){$\dots$}
\put(200,0){\circle{6}}
\put(240,0){\circle{6}}
\put(35,-15){$\varepsilon_1-\varepsilon_2$}
\put(75,-15){$\varepsilon_2-\varepsilon_3$}
\put(175,-15){$\varepsilon_{n-1}-\varepsilon_{n}$}
\put(225,-15){$\varepsilon_n-\varepsilon_{n+1}$}
\put(53,0){\line(1,0){34}}
\put(93,0){\line(1,0){34}}
\put(203,0){\line(1,0){34}}
\end{picture}
\vspace{11 mm}

\noindent Then the positive roots with the above fundamental system is
$$\Phi^+ = \{ \varepsilon_i- \varepsilon_{j}, \ 1\leq i<j \leq n+1 \}.$$

\smallskip
\noindent (2) The fundamental system of type $C_n$ is

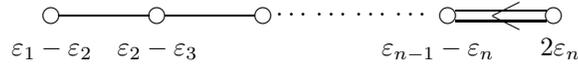
\begin{picture}(40,10)(-45,10)
\put(50,0){\circle{6}}
\put(90,0){\circle{6}}
\put(130,0){\circle{6}}
\put(135,0){$\dots$}
\put(150,0){$\dots$}
\put(165,0){$\dots$}
\put(180,0){$\dots$}
\put(200,0){\circle{6}}
\put(240,0){\circle{6}}
\put(35,-15){$\varepsilon_1-\varepsilon_2$}
\put(75,-15){$\varepsilon_2-\varepsilon_3$}
\put(175,-15){$\varepsilon_{n-1}-\varepsilon_{n}$}
\put(235,-15){$2\varepsilon_n$}
\put(53,0){\line(1,0){34}}
\put(93,0){\line(1,0){34}}
\put(203,-2){\line(1,0){34}}
\put(203,2){\line(1,0){34}}
\put(217,0){\line(2,1){10}}
\put(217,0){\line(2,-1){10}}
\end{picture}
\vspace{11 mm}

\noindent Then the positive roots with the above fundamental system is
$$\Phi^+ = \{ \varepsilon_i\pm \varepsilon_{j}, 1\leq i<j\leq n;\  2\varepsilon_k , 1\leq k \leq n \}.$$

\smallskip

\noindent (3) The fundamental system of type $B_n$ is

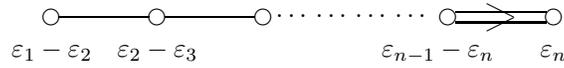
\begin{picture}(40,10)(-45,10)
\put(50,0){\circle{6}}
\put(90,0){\circle{6}}
\put(130,0){\circle{6}}
\put(135,0){$\dots$}
\put(150,0){$\dots$}
\put(165,0){$\dots$}
\put(180,0){$\dots$}
\put(200,0){\circle{6}}
\put(240,0){\circle{6}}
\put(35,-15){$\varepsilon_1-\varepsilon_2$}
\put(75,-15){$\varepsilon_2-\varepsilon_3$}
\put(175,-15){$\varepsilon_{n-1}-\varepsilon_{n}$}
\put(235,-15){$\varepsilon_n$}
\put(53,0){\line(1,0){34}}
\put(93,0){\line(1,0){34}}
\put(203,2){\line(1,0){34}}
\put(203,-2){\line(1,0){34}}
\put(225,0){\line(-2,1){10}}
\put(225,0){\line(-2,-1){10}}
\end{picture}
\vspace{11 mm}

\noindent Then the positive roots with the above fundamental system is
$$\Phi^+ = \{ \varepsilon_i\pm \varepsilon_{j}, 1\leq i<j\leq n;\  \varepsilon_k , 1\leq k \leq n \}.$$

\smallskip

\noindent (4) The fundamental system of type $D_n$ is

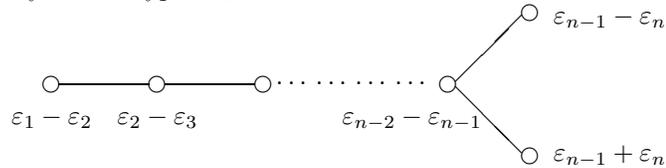
\begin{picture}(40,10)(-45,20)
\put(50,0){\circle{6}}
\put(90,0){\circle{6}}
\put(130,0){\circle{6}}
\put(135,0){$\dots$}
\put(150,0){$\dots$}
\put(165,0){$\dots$}
\put(180,0){$\dots$}
\put(200,0){\circle{6}}
\put(231,27){\circle{6}}
\put(231,-27){\circle{6}}
\put(35,-15){$\varepsilon_1-\varepsilon_2$}
\put(75,-15){$\varepsilon_2-\varepsilon_3$}
\put(160,-15){$\varepsilon_{n-2}-\varepsilon_{n-1}$}
\put(240,23){$\varepsilon_{n-1}-\varepsilon_{n}$}
\put(240,-29){$\varepsilon_{n-1}+\varepsilon_{n}$}
\put(53,0){\line(1,0){34}}
\put(93,0){\line(1,0){34}}
\put(203,1){\line(1,1){25}}
\put(203,-1){\line(1,-1){25}}
\end{picture}
\vspace{20 mm}

\noindent Then the positive roots with the above fundamental system is
$$\Phi^+ = \{ \varepsilon_i\pm \varepsilon_{j},\ 1\leq i<j\leq n \}.$$

\setcounter{theorem}{2}
\subsection{Weighted Dynkin diagram} \label{WDD}
Let $G_0$ be a connected reductive algebraic group over $\mathbb C$ of the same type as $G$ and let $\mathfrak g_0$ be its Lie algebra. Then by the classical Dynkin-Kostant theory, the nilpotent Ad$(G_0)$-orbits in $\mathfrak g_0$ are parameterized by weighted Dynkin diagrams.
If $G_0$ is a simple algebraic group, then the corresponding set $\Delta$ of weighted Dynkin diagrams is explicitly known in all cases
.
For convenience, we put here, the descriptions of the set $\Delta$ of weighted Dynkin diagrams of classical types,
and the set $\Delta_{\text{spec}}$ of special weighted Dynkin diagrams of types $B_n,C_n$ and $D_n$, which are given in \cite[\S 13.1]{C}.

\smallskip

\noindent \textbf{Type $A_n$}: The weighted Dynkin diagrams are 
parametrized by $\mathscr P_{n+1}$ when the group $G_0$ is simple of type $A_n$. 
Let $\mu=(\mu_1,\mu_2,\dots,\mu_k) \in \mathscr P_{n+1}$. For each elementary division $\mu_i$, we take the set of integers $\mu_i-1, \mu_i-3, \dots, 3-\mu_i,1-\mu_i$.
We take the union of these sets for all elementary divisors and write this union as $\{\xi_1,\xi_2,\dots, \xi_{n+1}\}$ with $\xi_1\geq \xi_2 \geq \dots \geq \xi_{n+1}$. Then the corresponding
weighted Dynkin diagram is

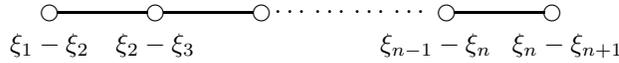
\begin{picture}(40,10)(-45,10)
\put(50,0){\circle{6}}
\put(90,0){\circle{6}}
\put(130,0){\circle{6}}
\put(135,0){$\dots$}
\put(150,0){$\dots$}
\put(165,0){$\dots$}
\put(180,0){$\dots$}
\put(200,0){\circle{6}}
\put(240,0){\circle{6}}
\put(35,-15){$\xi_1-\xi_2$}
\put(75,-15){$\xi_2-\xi_3$}
\put(175,-15){$\xi_{n-1}-\xi_{n}$}
\put(225,-15){$\xi_n-\xi_{n+1}$}
\put(53,0){\line(1,0){34}}
\put(93,0){\line(1,0){34}}
\put(203,0){\line(1,0){34}}
\end{picture}
\vspace{11 mm}

\smallskip

\noindent \textbf{Type $C_n$}:
The weighted Dynkin diagrams are parametrized by the pairs of partitions $(\mu, \nu)$ with $|\mu|+|\nu|=n$ where $\nu$ has distinct parts
.
For such a partition $(\mu, \nu)$, write $\mu=(\mu_1,\mu_2, \cdots, \mu_k)$,
$\nu=(\nu_1,\nu_2,\cdots,\nu_l)$ and 
consider the elementary divisors $\mu_1, \mu_1, \mu_2, \mu_2,\dots, \mu_k, \mu_k$, together with $2\nu_1,2\nu_2,\dots, 2\nu_l$.
Arrange them as an ordered sequence $\mathfrak d=(m_1,m_2,\dots)$, where $m_1 \leq m_2 \leq \cdots$.
For each elementary divisor $m$
in $\mathfrak d$, we take the set of integers $m-1, m-3, \dots, 3-m,1-m$. We take the union of these sets for all elementary divisors and write this union as $\{\xi_1,\xi_2,\dots, \xi_{2n}\}$ with $\xi_1\geq \xi_2 \geq \dots \geq \xi_{2n}$.
Then the associated weighted Dynkin diagram is

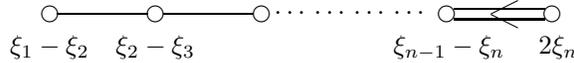
\begin{picture}(40,10)(-45,10)
\put(50,0){\circle{6}}
\put(90,0){\circle{6}}
\put(130,0){\circle{6}}
\put(135,0){$\dots$}
\put(150,0){$\dots$}
\put(165,0){$\dots$}
\put(180,0){$\dots$}
\put(200,0){\circle{6}}
\put(240,0){\circle{6}}
\put(35,-15){$\xi_1-\xi_2$}
\put(75,-15){$\xi_2-\xi_3$}
\put(180,-15){$\xi_{n-1}-\xi_n$}
\put(235,-15){$2 \xi_n$}
\put(53,0){\line(1,0){34}}
\put(93,0){\line(1,0){34}}
\put(203,-2){\line(1,0){34}}
\put(203,2){\line(1,0){34}}
\put(217,0){\line(2,1){10}}
\put(217,0){\line(2,-1){10}}
\end{picture}
\vspace{11 mm}

Denote this weighted Dynkin diagram as $d \in \Delta$.
Then $d$ is special if and only if the following condition is satisfied:

\noindent ($\spadesuit_C$) Between any two consecutive odd elementary divisors $2i + 1$,  $2j + 1$ in $\mathfrak{d}$ with $i < j$, there is an even number of even divisors, and after the largest odd divisor there is
an even number of even divisors.

\smallskip

\noindent \textbf{Type $B_n$}:
The weighted Dynkin diagrams
are parametrized by the pairs of partitions $(\mu, \nu)$ with $2|\mu|+|\nu|=2n+1$ where $\nu$ has distinct odd parts
. For such a partition $(\mu, \nu)$, write $\mu=(\mu_1,\mu_2, \cdots, \mu_k)$,
$\nu=(\nu_1,\nu_2,\cdots,\nu_l)$ and 
consider the elementary divisors $\mu_1, \mu_1, \mu_2, \mu_2,\dots, \mu_k, \mu_k$, together with $\nu_1,\nu_2,\dots, \nu_l$.
Arrange them as an ordered sequence $\mathfrak d=(m_1,m_2,\dots)$, where $m_1 \leq m_2 \leq \cdots$.
For each elementary divisor $m$ in $\mathfrak d$, we take the set of integers $m-1, m-3, \dots, 3-m,1-m$. We take the union of these sets for all elementary divisors and write this union as $\{\xi_1,\xi_2,\dots, \xi_{2n+1}\}$ with $\xi_1\geq \xi_2 \geq \dots \geq \xi_{2n+1}$.
Then the associated weighted Dynkin diagram is

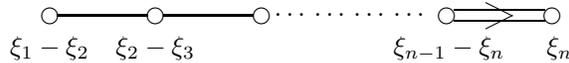
\begin{picture}(40,10)(-45,10)
\put(50,0){\circle{6}}
\put(90,0){\circle{6}}
\put(130,0){\circle{6}}
\put(135,0){$\dots$}
\put(150,0){$\dots$}
\put(165,0){$\dots$}
\put(180,0){$\dots$}
\put(200,0){\circle{6}}
\put(240,0){\circle{6}}
\put(35,-15){$\xi_1-\xi_2$}
\put(75,-15){$\xi_2-\xi_3$}
\put(180,-15){$\xi_{n-1}-\xi_n$}
\put(238,-15){$\xi_n$}
\put(53,0){\line(1,0){34}}
\put(93,0){\line(1,0){34}}
\put(203,2){\line(1,0){34}}
\put(203,-2){\line(1,0){34}}
\put(225,0){\line(-2,1){10}}
\put(225,0){\line(-2,-1){10}}
\end{picture}
\vspace{11 mm}

Denote this weighted Dynkin diagram as $d \in \Delta$.
Then $d$ is special if and only if the following condition is satisfied:

\noindent ($\spadesuit_B$) Between any two consecutive even elementary divisors of $\mathfrak{d}$ there is an even number of odd divisors, and after the largest even divisors there is an odd number of odd divisors.

\smallskip

\noindent \textbf{Type $D_n$}:
The weighted Dynkin diagrams 
are parametrized by the pairs of partitions $(\mu, \nu)$ with $2|\mu|+|\nu|=2n$ where $\nu$ has distinct odd parts
. For such a partition $(\mu, \nu)$, write $\mu=(\mu_1,\mu_2, \cdots, \mu_k)$,
$\nu=(\nu_1,\nu_2,\cdots,\nu_l)$ and 
consider the elementary divisors $\mu_1, \mu_1, \mu_2, \mu_2,\dots, \mu_k, \mu_k$, together with $\nu_1,\nu_2,\dots, \nu_l$.
Arrange them as an ordered sequence $\mathfrak d=(m_1,m_2,\dots)$, where $m_1 \leq m_2 \leq \cdots$.
For each elementary divisor $m$ in $\mathfrak d$, we take the set of integers $m-1, m-3, \dots, 3-m,1-m$. We take the union of these sets for all elementary divisors and write this union as $\{\xi_1,\xi_2,\dots, \xi_{2n}\}$ with $\xi_1\geq \xi_2 \geq \dots \geq \xi_{2n}$.
Then the associated weighted Dynkin diagram is
\begin{eqnarray} \label{d_D1}
\text{
\begin{picture}(40,10)(130,20)
\put(50,0){\circle{6}}
\put(90,0){\circle{6}}
\put(130,0){\circle{6}}
\put(135,0){$\dots$}
\put(150,0){$\dots$}
\put(165,0){$\dots$}
\put(180,0){$\dots$}
\put(200,0){\circle{6}}
\put(231,27){\circle{6}}
\put(231,-27){\circle{6}}
\put(35,-15){$\xi_1-\xi_2$}
\put(75,-15){$\xi_2-\xi_3$}
\put(160,-15){$\xi_{n-2}-\xi_{n-1}$}
\put(240,23){$\xi_{n-1}-\xi_{n}$}
\put(240,-29){$\xi_{n-1}+\xi_{n}$}
\put(53,0){\line(1,0){34}}
\put(93,0){\line(1,0){34}}
\put(203,1){\line(1,1){25}}
\put(203,-1){\line(1,-1){25}}
\end{picture}}
\end{eqnarray}
\vspace{20 mm}

\noindent or

\begin{eqnarray} \label{d_D2}
\text{
\begin{picture}(40,10)(130,20)
\put(50,0){\circle{6}}
\put(90,0){\circle{6}}
\put(130,0){\circle{6}}
\put(135,0){$\dots$}
\put(150,0){$\dots$}
\put(165,0){$\dots$}
\put(180,0){$\dots$}
\put(200,0){\circle{6}}
\put(231,27){\circle{6}}
\put(231,-27){\circle{6}}
\put(35,-15){$\xi_1-\xi_2$}
\put(75,-15){$\xi_2-\xi_3$}
\put(160,-15){$\xi_{n-2}-\xi_{n-1}$}
\put(240,23){$\xi_{n-1}+\xi_{n}$}
\put(240,-29){$\xi_{n-1}-\xi_{n}$}
\put(53,0){\line(1,0){34}}
\put(93,0){\line(1,0){34}}
\put(203,1){\line(1,1){25}}
\put(203,-1){\line(1,-1){25}}
\end{picture}}
\end{eqnarray}
\vspace{20 mm}

If there is some odd part in $(\mu,\nu)$ then $\xi_n=0$ and so the diagrams are the same. If $\nu$ is empty and all parts of $\mu$ are even then $\xi_n\ne 0$  and we obtained two weighted Dynkin diagrams associate to the same $(\mu,\nu)$.

Denote this weighted Dynkin diagram as $d \in \Delta$.
Then $d$ is special if and only if the following condition is satisfied:

\noindent ($\spadesuit_D$) Between any two consecutive even elementary divisors of $\mathfrak{d}$ there is an even number of odd divisors, and after the largest even divisor there is an even number of odd divisors.

\subsection{Gram matrix} Let $d\in \Delta$ and assume that $\mathfrak g_d(1) \neq \{0\}$. (So $\mathfrak g_{\mathbb Z,d}(1) \neq \{0\}$.)
Given a linear map $\lambda:\mathfrak g_d(2)\rightarrow \mathbf k$ (resp. $\lambda: \mathfrak g_{\mathbb Z,d}(2) \rightarrow \mathbb Z$), we obtain an alternating form $\sigma_{\lambda}: \mathfrak g_d(1) \times \mathfrak g_d(1) \rightarrow \mathbf k$ (resp. $\sigma_{\lambda }: \mathfrak g_{\mathbb Z,d}(1)\times \mathfrak g_{\mathbb Z,d}(1) \rightarrow \mathbb Z$), and as in Section 0, we denote its Gram matrix as $\mathscr G_{d,\lambda}$.

 Write $\Phi_{d,1}=\{\beta_1,\beta_2, \dots, \beta_k \}$ and $\Phi_{d,2}=\{\gamma_1,\gamma_2,\dots,\gamma_m\}$.
The entries of $\mathscr G_{d,\lambda}$ are given as follows.
We set $x_l:=\lambda(e_{\gamma_l})$ for $1\leq l\leq m$. For $1\leq i, j\leq k$, we define an element $\nu_{ij}\in \mathbf k$ (resp. $\nu_{ij}\in \mathbb Z$) as follows. If $\beta_i+\beta_j \notin \Phi$, then $\nu_{ij}=0$. Otherwise, there is a unique $l(i,j)\in \{1,2,\dots,m\}$ and some $\nu_{ij}\in \mathbf k$ (resp. $\nu_{ij}\in \mathbb Z$ ) such that $[e_{\beta_i},e_{\beta_j}]=\nu_{ij}e_{\gamma_{l(i,j)}}$. Then we have
\begin{eqnarray}
(\mathscr G_{d,\lambda})_{ij} =\sigma_{\lambda}(e_{\beta_i},e_{\beta_j})=
\begin{cases}
x_{l(i,j)}\nu_{ij} & \text{ if } \beta_i+\beta_j \in \Phi, \\

0 &  \ \text{otherwise.}
\end{cases} \nonumber
\end{eqnarray}

For convenience of later computation, we can choose  $e_{\alpha}$  in $\mathfrak g_{\alpha}$ for each $\alpha \in \Phi$, which satisfies:
\begin{eqnarray}
[e_{\alpha}, e_{\beta} ] =
\begin{cases}
0 & \text{ if } \alpha+ \beta \not\in \Phi, \\

\pm(r+1)e_{\alpha+\beta}&  \text{ if } \alpha+ \beta \in \Phi,
\end{cases} \nonumber
\end{eqnarray}
where $r$ is the greatest integer such that $\beta-r\alpha\in \Phi$.
For the type $A$ case, such integer $r$ is always zero. For other classical cases (type $B,C,D$), such integer $r\in\{0,1\}$.

\subsection{Reduction of the main theorem} \label{reduction}

To prove Theorem \ref{mainThm2},
we do some reduction first. Put $\Delta^o =\{ d \in \Delta \mid  \displaystyle \max_{\alpha \in \Pi}d(\alpha) =1 \}$. The elements in $\Delta^o$ are called \emph{odd weighted Dynkin diagrams}.
We show that to prove Theorem \ref{mainThm1}, it is enough to consider the odd weighted Dynkin diagrams.

For a fixed simple root system $\Pi$ of classical type and given a partition (type $A$) or a bipartition (type $B$, $C$, $D$) associated to a unipotent class, we have a corresponding set of elementary divisors
$\mathfrak{d}$. Then we can construct a  weighted Dynkin diagram $d\in \Delta$ from $\mathfrak{d}$. Now we assume that $d\in \Delta\setminus (\Delta^o \cup 0 ) $, i.e. 
$d(\alpha)=2$ for some $\alpha \in \Pi$.

Suppose that $\mathfrak{d}=(1^{m_1}2^{m_2}\dots r^{m_r})$ with $m_r\ne 0$. Since $d\in \Delta\setminus (\Delta^o \cup 0 ) $, the maximal integer $s<r$ such that $m_s \ne 0$ has to satisfy $s<r-1$. Then we set $t=r-2 \lfloor\frac{r-s}{2} \rfloor$ and $m_t=m_r$. Now we consider a new set of elementary divisors
$\mathfrak{d}'= (1^{m_1}2^{m_2}\dots s^{m_s}t^{m_t})$.  By its construction, it is not difficult to check that the set of elementary divisors $\mathfrak{d}'$ comes from an appropriate partition (or a bipartition).
Thus $\mathfrak{d}'$ also gives a weighted Dynkin diagram $d'\in \Delta'$ of a simple root system whose rank is smaller than the original simple root system.
Therefore to study the Gram matrix  $\mathscr G_{d,\lambda}$ can be reduced to study the Gram matrix
$\mathscr G_{d',\lambda'}$ where $\lambda'$ is the restriction of  $\lambda$. Moreover $d\in \Delta$ is special if and only if $d'\in \Delta'$ is special. Such reduction process can go on until $d$ becomes an odd weighted Dynkin diagram or $d(\Pi)=\{0\}$.
When $d(\Pi)=\{0\}$, $d$ is special and $\mathfrak g_{d}(1)=\{0\}$.
Thus Theorem \ref{mainThm1} will be implied by the following Theorem.

\setcounter{theorem}{4}
\begin{theorem} \label{mainThm2}
Assume that $G$ is of classical type $A_n$, $B_n$, $C_n$ or $D_n$. Let $d \in \Delta^o$, which implies that $\mathfrak g_{d}(1) \neq \{0\}$.
\begin{enumerate}
\item If $d \in \Delta_{\text{spec}}$, then there exists a
    non-degenerate $\lambda \in \operatorname{Hom}_{\mathbb Z}(\mathfrak g_{\mathbb Z,d}(2), \mathbb Z)$.
\item  If $d \not\in \Delta_{\text{spec}}$ and $\operatorname{char} \mathbf k =2$, then $\det \mathscr G_{d,\lambda} = 0$ for any $\lambda \in \operatorname{Hom}_{\mathbf k}(\mathfrak g_d(2), \mathbf k)$.
\end{enumerate}
\end{theorem}

In the following sections we will prove Theorem \ref{mainThm2} through case by case discussions. For (1), we will construct a non-degenerate $\lambda \in \operatorname{Hom}_{\mathbb Z}(\mathfrak g_{\mathbb Z,d}(2), \mathbb Z)$ explicitly when $d \in \Delta_{\text{spec}}$.
For (2), we need to calculate the determinant of symmetric matrices over  $\mathbf k$ when $\text{char}~\mathbf k =2 $; the properties of this type of matrices will be given in the next section.

\section{Matrices over a field of even characteristic}
The only bad prime is $2$ for classical groups of types $B,C$ and $D$. So in this section we study the determinant of certain symmetric  matrices over a field of even characteristic. 
Throughout this section, we assume that $\operatorname{ch} \mathbf k=2$.

For $A\in M_{m\times n}(\mathbf k )$ and $B\in M_{l\times n}(\mathbf k )$ with $m\leq l$, we denote the rows of $A$ by $A_1, A_2, \dots A_m$ and the rows of $B$ by $B_1,B_2,\dots, B_l$.
Denote $A\prec B$ if 
there exists an injective map $\tau: [1,m] \rightarrow [1,l]$, satisfying $A_i = B_{\tau(i)}$ for $i \in [1,m]$.

We give the following lemma, whose proof is easy, and thus omitted here.

\begin{lemma}\label{evendegSM}
Let $m \in 2\mathbb N$, and $A \in M_m(\mathbf k)$. Assume that $A$ is symmetric with all diagonal elements zero and that $\det A=0 $. Let $X=\{x\in \mathbf k^m \mid A{^tx}=0\}$. Then $\dim X \geq 2$.
\end{lemma}

\begin{proposition} \label{p_det_D}

Let $n, N \in \mathbb Z_+$ and $\Gamma\in \{2n, 2n+1\}$. For $r \in [1,N]$, let $k_r \in \mathbb Z_+$, $m_r \in [1,2n]$, and $A_r \in M_{m_r \times \Gamma}(\mathbf k)$. Assume that
$m_r <m_s$ and $A_r \prec A_s$ for $r <s$ in $[1,N]$. Let
$E^{r,s}_{i,j}\in M_{\Gamma}(\mathbf k )$, for $r,s \in [1,N]$, $i \in [1, k_r]$, $j \in [1, k_s]$, which is of the form
\begin{eqnarray}
E^{r,s}_{i,j} =
\begin{cases}
\begin{pmatrix}
& 0          & e^{r,s}_{i,j} I_n\\
& e^{r,s}_{i,j} I_n            &0\\
\end{pmatrix} & \text{ if } \Gamma=2n, \\
\begin{pmatrix}
& 0  & 0  & 0\\
& 0 & 0          & e^{r,s}_{i,j} I_n\\

& 0 & e^{r,s}_{i,j} I_n            &0\\
\end{pmatrix} & \text{ if } \Gamma=2n+1,
\end{cases} \nonumber
\end{eqnarray}
where $e^{r,s}_{i,j}= e^{s,r}_{j,i} \in \mathbf k$.
Let
\begin{eqnarray}
S_{i,i} & = & \left(\begin{array}{cc; {1pt/1pt}cc; {1pt/1pt}cc; {1pt/1pt}cc}
   0  & A_{i,1} &  0      &   0   & \cdots  & \cdots  &0 &0         \\
  {^t A_{i,1}}   & 0   & 0   &  E^{i,i}_{1,2}   & \cdots  & \cdots  &0 & E^{i,i}_{1,k_i} \\
  \hdashline[1pt/1pt]
  0  & 0   & 0   & A_{i,2}    &  \dots  & \dots  &  0   &0 \\
0  & E^{i,i}_{2,1}  & {^t A_{i,2} }    & 0  & \cdots     &  \dots    & 0  & E^{i,i}_{2,k_i}  \\
\hdashline[1pt/1pt]
\vdots  & \vdots   & \vdots   & \vdots     &  \vdots  & \vdots  &  \vdots  & \vdots  \\
\vdots  & \vdots   & \vdots   & \vdots     &  \vdots  & \vdots  &  \vdots  & \vdots  \\
\hdashline[1pt/1pt]
0 & 0    &  0 & 0 &  \dots  & \dots   &  0 & A_{i,k_i} \\
0 & E^{i,i}_{k_i,1}   & 0  &E^{i,i}_{k_i,2}  &  \dots  & \dots   & {^t  A_{i,k_i}} &  0 \\
\end{array}\right) \quad \text{ for } i \in [1,N], \nonumber \\
S_{i,j} & = &  \left(\begin{array}{cc; {1pt/1pt}cc; {1pt/1pt}cc; {1pt/1pt}cc}
   0  & 0 &  0      &   0   & \cdots  & \cdots  &0 &0         \\
  0   & E^{i,j}_{1,1}   & 0   &  E^{i,j}_{1,2}   & \cdots  & \cdots  &0 & E^{i,j}_{1,k_j} \\
  \hdashline[1pt/1pt]
  0  & 0   & 0   & 0    &  \dots  & \dots  &  0   &0 \\
0  & E^{i,j}_{2,1}  & 0   & E^{i,j}_{2,2} & \cdots     &  \dots    & 0  & E^{i,j}_{2,k_j}  \\
\hdashline[1pt/1pt]
\vdots  & \vdots   & \vdots   & \vdots     &  \vdots  & \vdots  &  \vdots  & \vdots  \\
\vdots  & \vdots   & \vdots   & \vdots     &  \vdots  & \vdots  &  \vdots  & \vdots  \\
\hdashline[1pt/1pt]
0 & 0    &  0 & 0 &  \dots  & \dots   &  0 & 0 \\
0 & E^{i,j}_{k_i,1}   & 0  &E^{i,j}_{k_i,2}  &  \dots  & \dots   & 0 &  E^{i,j}_{k_i,k_j} \\
\end{array}\right) \quad \text{ for } i\neq j \text{ in }[1,N], \nonumber
\end{eqnarray}
where $A_{r,i}=A_r$ for $i\in [1, k_r]$, and $S$ be the blocked matrix $(S_{i,j})_{i,j \in [1,N]}$.
Then $\det S=0$ if one of the following conditions is satisfied:
\begin{enumerate}
\item $\Gamma = 2n$, $m_r$ is even for $r \in [1,N]$ and there exists $\sigma \in [1,N]$ with $k_{\sigma}$ odd and $m_{\sigma} <2n$;
\item $\Gamma= 2n$, and there exits $\sigma \in [1,N]$ with $m_{\sigma}$  odd;
\item $\Gamma =2n+1$, and $m_{\sigma}$ is even for some $\sigma \in [1,N]$.
\end{enumerate}
\end{proposition}

\begin{proof}
Let $Z$ be a row vector of proper size, satisfying $S{^tZ} =0$. It is enough to show that $Z$ can be taken nonzero.

\noindent  (1) For $r \in [1,N]$ and $i \in [1,k_r]$, write $A_{ri} = (X_{ri},Y_{ri})$, where $X_{ri},Y_{ri} \in M_{m_r \times n} (\mathbf k)$; as $A_{r,i} =A_r$,
we can denote $X_{ri}$, $Y_{r,i}$ as $X_r$, $Y_r$ respectively.
Write $Z=(Z_1, \cdots, Z_N)$, with $$Z_r=(T_{r1},U_{r1},V_{r1}, \cdots, T_{r,k_r},
U_{r,k_r},V_{r,k_r}),$$ where $T_{ri} \in \mathbf k^{m_r}, U_{ri},V_{ri} \in \mathbf k^{n}$ are row vectors, for $r \in [1,N]$ and $i \in [1,k_r]$. It is clear that, $S{^tZ}=0$ is equivalent to the following system of linear equations:
\begin{equation}\label{equC}
\left\{
\begin{array}{lcc}
X_{ri} {^tU_{ri}}+\displaystyle  Y_{ri} {^tV_{ri}}&=&0, \\
^t X_{ri} {^tT_{ri}}+\displaystyle \sum_{s,j} e^{r,s}_{i,j} ({^tV_{sj}})&=&0,  \\
^t Y_{ri} {^tT_{ri}}+\displaystyle \sum_{s,j} e^{r,s}_{i,j} ({^tU_{sj}})&=&0,
\end{array}\right.
\end{equation}
for $r \in [1,N]$ and $i \in [1,k_r]$.
Without loss of generality, in the following we assume that $\sigma$ is the minimal integer such that $k_{\sigma}$ is odd and $m_{\sigma}< 2n$ .

We let
$$E = \begin{pmatrix}
E^{1,1} & \cdots & E^{1,\sigma} \\
\cdots &\cdots & \cdots \\
E^{\sigma,1} & \cdots & E^{\sigma,\sigma}
\end{pmatrix},$$
where
$$E^{i,j} = \begin{pmatrix}
e^{i,j}_{1,1} & \cdots & e^{i,j}_{1,k_j} \\
\cdots & \cdots & \cdots \\
e^{i,j}_{k_i,1} & \cdots & e^{i,j}_{k_i,k_j}
\end{pmatrix},$$
for $i,j \in [1,\sigma]$ and $e^{i,i}_{l,l}=0$ for $i\in [1, \sigma], l\in [1,k_i]$.
Since $m_{\sigma}<2n$,
there exist row vectors $U_{\sigma}, V_{\sigma} \in \mathbf k^n$, not both $0$, such that $X_{\sigma}{}^t U_{\sigma}+ Y_{\sigma} {}^tV_{\sigma}=0$.
Since $k_{\sigma}$ is odd, $E$ is a symmetric matrix of odd degree with diagonal elements zero. As $\operatorname{ch}\mathbf k =2$, there exists a row vector
$$W =(w_{11}, \cdots, w_{1,k_1}, \cdots, w_{\sigma,1}, \cdots, w_{\sigma,k_{\sigma}}) \in \mathbf k^{k_1 + \cdots +k_{\sigma}},$$
satisfying $E{^tW}=0$.

Now we set
\begin{eqnarray}
\begin{array}{ll}
 T_{r,i} = 0 & \text{for }r \in [1,N], i \in [1,k_r],  \\
 U_{r,i} = V_{r,i}=0 & \text{for } r \in [\sigma+1, N], i \in [1,k_r], \\
 U_{\sigma,i} = U_{\sigma}, V_{\sigma,i} =V_{\sigma} & \text{for } i \in [1,k_{\sigma}], \\
 U_{r,i} = w_{r,i}U_{\sigma}, V_{r,i} = w_{r,i} V_{\sigma} & \text{for } r \in [1,\sigma-1], i \in [1,k_r].
 \end{array} \nonumber
\end{eqnarray}
The vector $Z$ thus obtained satisfies $Z \neq 0$ and $S{^tZ}=0$.

\medskip
\noindent (2) Completely as in (1), we get the equivalence between $S{^tZ}=0$ and the system of linear equations (\ref{equC}).
Since $ X_{\sigma} {}^t Y_{\sigma}+ Y_{\sigma} {}^t X_{\sigma} $ is a symmetric matrix with diagonals zero, and its rank $m_{\sigma}$ is odd, there exists a row vector $W \in \mathbf k^{m_{\sigma}}$ satisfying $W \neq 0$ and $$(X_{\sigma} {}^t Y_{\sigma}+ Y_{\sigma} {}^t X_{\sigma}){^tW} =0.$$
We set $$U_{\sigma,i} = U_{\sigma} = W {Y_{\sigma}}\  \text{and} \  V_{\sigma,i}= V_{\sigma}= W{X_{\sigma}} \ \text{for} \ i \in [1,k_{\sigma}].$$

If $U_{\sigma} = V_{\sigma}=0$, then set
\begin{eqnarray}
\begin{array}{ll}
 T_{r,i} = 0 & \text{for }r \in [1,N], r \neq \sigma, i \in [1,k_r],  \\
 T_{\sigma,i} = W & \text{for } i \in [1,k_{\sigma}], \\
 U_{r,i} = V_{r,i} =0  & \text{for } r \in [1, N], i \in [1,k_r].
 \end{array} \nonumber
\end{eqnarray}
The matrix $Z$ thus obtained is nonzero and satisfies $S{^tZ}=0$.

If $U_{\sigma} \neq 0$ or $V_{\sigma} \neq 0$, then put
$$E = \begin{pmatrix}
E^{1,1} & \cdots & E^{1,\sigma} \\
\cdots & \cdots & \cdots \\
E^{\sigma-1,1} & \cdots & E^{\sigma-1,\sigma}
\end{pmatrix},$$
where
$$E^{i,j} = \begin{pmatrix}
e^{i,j}_{1,1} & \cdots & e^{i,j}_{1,k_j} \\
\cdots & \cdots & \cdots \\
e^{i,j}_{k_i,1} & \cdots & e^{i,j}_{k_i,k_j}
\end{pmatrix}$$
for $i \in [1,\sigma-1]$ and $j \in [1, \sigma]$.
Since $E$ is not of full column rank, there exists
$$W'=(w_{1,1}, \cdots, w_{1,k_1}, \cdots, w_{\sigma,1}, \cdots, w_{\sigma,k_{\sigma}}) \in \mathbf k^{k_1+ \cdots + k_{\sigma}}$$ satisfying $W' \neq 0$ and $E{^tW'}=0$.

Let
$d_{r,i} = \sum_{s \in [1,\sigma],j \in [1,k_s]}e^{r,s}_{i,j}w_{s,j}$
for $r \in [\sigma, N]$ and $i \in [1,k_r]$.
For $r \in [\sigma, N]$, since $X_{\sigma} \prec X_r$, through adding a certain number of $0$'s in the end of $W$, we can get a vector $W'' \in \mathbf k^{m_r}$ satisfying
$${^tX_r}{^tW''} + {^tX_{\sigma}}{^tW} = 0.$$
Set
\begin{eqnarray}
\begin{array}{ll}
 T_{r,i} = 0 & \text{for }r \in [1,\sigma-1], i \in [1,k_r],  \\
 T_{r,i} = d_{r,i}W'' & \text{for } r \in [\sigma,N], i \in [1,k_r], \\
 U_{r,i} = V_{r,i} =0  & \text{for } r \in [\sigma+1, N], i \in [1,k_r], \\
 U_{r,i} = w_{r,i}U_{\sigma}, V_{r,i} = w_{r,i}V_{\sigma}, & \text{for }
 r \in [1, \sigma], i \in [1,k_r].
 \end{array} \nonumber
\end{eqnarray}
The matrix $Z$ thus obtained satisfies $Z \neq 0$ and $S{^tZ}=0$

\medskip
\noindent (3) For $r \in [1,N]$ and $i \in [1,k_r]$, write $A_{ri} = (P_{ri},X_{ri},Y_{ri})$, where $X_{ri},Y_{ri} \in
M_{m_r \times n} (\mathbf k)$, $P_{ri} \in M_{m_r \times 1}(\mathbf k)$; as $A_{r,i} =A_r$,
we can denote $X_{r,i}$, $Y_{r,i}$, $P_{r,i}$ as $X_r$, $Y_r$, $P_r$ respectively.
Write $Z=(Z_1, \cdots, Z_N)$, with $$Z_r=(T_{r1},\chi_{r1}, U_{r1},V_{r1}, \cdots, T_{r,k_r}, \chi_{r,k_r}, U_{r,k_r},V_{r,k_r}),$$ where $T_{ri} \in \mathbf k^{m_r}, U_{ri},V_{ri} \in \mathbf k^{n}$ are row vectors, and $\chi_{ri} \in \mathbf k$ for $r \in [1,N]$ and $i \in [1,k_r]$. It is clear that, $S{^tZ}=0$ is equivalent to the following system of linear equations:

\begin{equation}\label{equB}
\left\{
\begin{array}{ll}
\chi_{ri}{^tP_{r,i}}+ X_{ri}{^t U_{ri}}+Y_{ri} {^tV_{ri}}=0 \\
^t X_{ri} {^tT_{ri}}+\displaystyle  \sum_{s,j} e^{r,s}_{i,j} ({^tV_{sj}})=0  \\
^t Y_{ri} {^tT_{ri}}+\displaystyle  \sum_{s,j} e^{r,s}_{i,j} ({^tU_{sj}})=0  \\
{}^t P_{ri} {^tT_{ri}}= 0
\end{array}\right.
\end{equation}
for $r \in [1,N]$ and $i \in [1,k_r]$.

If $X_{\sigma} {}^t Y_{\sigma}+ Y_{\sigma} {}^t X_{\sigma}$ is invertible, we set
$$\eta_{\sigma}= (X_{\sigma} {}^t Y_{\sigma}+ Y_{\sigma} {}^t X_{\sigma})^{-1} {^t Z_{\sigma}}.$$
We let $U_{\sigma}={}^t \eta_{\sigma} Y_{\sigma}   $ and $V_{\sigma}={}^t \eta_{\sigma} X_{\sigma} $.
Since $(X_{\sigma} {}^t Y_{\sigma}+ Y_{\sigma} {}^t X_{\sigma})^{-1}$ is a symmetric matrix with diagonal elements zero, it is easy to check that $Z_{\sigma} \eta_{\sigma} =0$.
As in the proof of (2), we put
$$E = \begin{pmatrix}
E^{1,1} & \cdots & E^{1,\sigma} \\
\cdots & \cdots & \cdots \\
E^{\sigma-1,1} & \cdots & E^{\sigma-1,\sigma}
\end{pmatrix},$$
where
$$E^{i,j} = \begin{pmatrix}
e^{i,j}_{1,1} & \cdots & e^{i,j}_{1,k_j} \\
\cdots & \cdots & \cdots \\
e^{i,j}_{k_i,1} & \cdots & e^{i,j}_{k_i,k_j}
\end{pmatrix}$$
for $i \in [1,\sigma-1]$ and $j \in [1, \sigma]$.
Since $E$ is not of full column rank, there exists
$$W'=(w_{1,1}, \cdots, w_{1,k_1}, \cdots, w_{\sigma,1}, \cdots, w_{\sigma,k_{\sigma}}) \in \mathbf k^{k_1+ \cdots + k_{\sigma}}$$ satisfying $W' \neq 0$ and $E  {^tW'}=0$.
Let
$d_{r,i} = \sum_{s \in [1,\sigma],j \in [1,k_s]}e^{r,s}_{i,j}w_{s,j}$
for $r \in [\sigma, N]$ and $i \in [1,k_r]$.
For $r \in [\sigma, N]$, since $X_{\sigma} \prec X_r$, through adding a certain number of $0$'s in the end of $\eta_{\sigma}$, we can get a row vector $W'' \in \mathbf k^{m_r}$ satisfying
$$ {^tX_{\sigma}} \eta_{\sigma} = {^tX_r}{^t W''} .$$
Set
\begin{eqnarray}
\begin{array}{ll}
\chi_{r,i} = w_{r,i} &\text{for } r \in [1, \sigma], i \in [1,k_r]\\
\chi_{r,i} = 0 & \text{for } r \in [\sigma+1, N], i \in [1,k_r]\\
U_{r,i} = w_{r,i}U_{\sigma}, V_{r,i} = w_{r,i}V_{\sigma}
& \text{for } r \in [1,\sigma], i \in [1,k_r], \\
U_{r,i} = V_{r,i} = 0 &  \text{for } r \in [\sigma+1,N], i \in [1,k_r], \\
T_{r,i} =0 & \text{for } r \in [1,\sigma-1], i \in [1,k_r], \\
T_{r,i} = d_{r,i}{W''} & \text{for } r \in [\sigma,N], i \in [1,k_r].
\end{array} \nonumber
\end{eqnarray}
The matrix $Z$ thus obtained satisfies $Z \neq 0$ and $S{^tZ}=0$.

If $X_{\sigma} {}^t Y_{\sigma}+ Y_{\sigma} {}^t X_{\sigma}$ is not invertible,
then the space
$$\{W \in  {\mathbf k}^{m_{\sigma}}  \mid (X_{\sigma} {}^t Y_{\sigma}+ Y_{\sigma} {}^t X_{\sigma}) {^t W} =0\}$$
is of dimension $\geq 2$ by Lemma \ref{evendegSM}, as $m_{\sigma}$ is even. Since $\text{char}~{\mathbf k}= 2$, we can find a nonzero row vector $W_\sigma$ such that
$$(X_{\sigma} {}^t Y_{\sigma}+ Y_{\sigma} {}^t X_{\sigma}) {^t W_\sigma} =0 \ \ \text{and} \ \ {}^t P_{\sigma}  {^t W_\sigma}  =0.$$ We set each $\chi_{ri}=0$,
then the proof follows exactly as the proof of (2).
Therefore the proposition is proved.

\end{proof}

\section{Type $A$}

\subsection{Odd weighted Dynkin diagram  }
In this section we deal with the type $A$ case. In this case all weighted Dynkin diagrams are special and thus it is enough  to prove  that  $\mathfrak g_{\mathbb Z,d}(1) =\{0\}$, or there exists a non-degenerate $\lambda \in \operatorname{Hom}_{\mathbb Z}( \mathfrak g_{\mathbb Z,d}(2),\mathbb Z)$ for each $d \in \Delta$.
Now assume that the simple Lie algebra $\mathfrak g$ is of type $A_n$ and all the simple roots are $\alpha_1,\alpha_2,\dots, \alpha_n$ and thus all positive roots can be denoted by  $\varepsilon_{i, j}=\alpha_i +\alpha_{i+1}+\dots+ \alpha_j$ with $1\leq i \leq  j\leq n$.

%

We can identify the  weighted Dynkin diagram $d\in \Delta^{o}$ 
with the sequence $$(i_r=0, i_{r-1}, i_{r-2},\cdots, i_1, i'_1,\cdots, i'_{r-2},i'_{r-1}, i'_r=n+1 )$$ of integers, where $\{i_{l},i'_{l} \mid l\in [1,r-1]\}=\{i\in [1,n] \mid d(\alpha_i)=1\}$,
satisfying the following conditions (1)-(3):
\begin{enumerate}
\item $1 \leq i_{r-1} <i_{r-2}< \cdots <i_1 <i' _1 <\cdots <i'_{r-2}< i'_{r-1} \leq n$,
\item $s_l=s'_l$ for $l=2,3,\cdots,r$,
\item $s_{l+2}\leq s_l$ for $1\leq l\leq r-2$,
\end{enumerate}
where $s_l=i_{l-1}-i_{l}$, $s'_l=i'_{l}-i'_{l-1}$ for $l=2,3,\cdots,r$ and $s_1=i'_1-i_1$.

\smallskip

In fact, suppose that $\mu=(1^{m_1}2^{m_2}\dots r^{m_r}) \in \mathscr P_{n+1}$, where $m_r\geq 1$, gives the weighted Dynkin diagram $d\in \Delta$ as above.
According to the discussion in Section \ref{reduction}, it is easy to see that $d\in \Delta^o$ if and only if  $m_{r-1}\geq 1$ in the partition $\mu$.

For this partition $\mu=(1^{m_1}2^{m_2}\dots r^{m_r})$, we set
\begin{equation*}a_l= \begin{cases}
m_l+m_{l+2}+\dots +m_r & \text{ if} \ r\equiv l
, \\
m_l+m_{l+2}+\dots +m_{r-1}  & \text{ if} \ r-1\equiv l 
.
\end{cases}
\end{equation*}
Then the sequence of integers $$(i_r=0, i_{r-1}, i_{r-2},\cdots, i_1, i'_1,\cdots, i'_{r-2},i'_{r-1}, i'_r=n+1 )$$
corresponding to the  weighted Dynkin diagram $d\in \Delta^{o}$ associated to $\mu$ satisfies that
$s_l=s'_l=a_l$ for $l=2,3,\dots, n$ and $s_1=a_1$. Using this idea, we can also identify a weighted Dynkin diagram with certain sequence of integers when we consider the type $B,C$ and $D$ cases later.

\subsection{Blocks of Gram matrix}
\setcounter{theorem}{2}
In order to prove Theorem \ref{mainThm2}, we have to look for a non-degenerate $\lambda \in \operatorname{Hom}_{\mathbb Z}(\mathfrak g_{\mathbb Z,d}(2), \mathbb Z)$. Our idea is to divide the set $\Phi_{d,1}$ into disjoint subsets and consider the submatrices of the Gram matrix corresponding to these subsets.
For the construction of equivalence classes, we define maps between subsets of $\Phi_{d,1}$.

Write $(i_r, i_{r-1}, \cdots, i_1,i_1', \cdots, i'_{r-1},i'_r)$ as
$(j_1,j_2, \cdots, j_r,j_{r+1}, \cdots, j_{2r-1},j_{2r})$. For $k \in [2,2r-2]$, we put
$$\mathcal{X}_k = \{ a\in [1,n] \ | \  j_{k-1}<a \leq j_k \ \text{and}\  a > j_k+j_{k+1}-j_{k+2}  \}, $$
and let $\varphi_k (a)= j_k+j_{k+1}-a $ for $a\in \mathcal{X}_k$. We set
$$X_k = \{ \varepsilon_{ab} | j_{k-1}<a \leq j_k \leq b <j_{k+1},
\ a \in  \mathcal{X}_k  \} \subseteq \Phi_{d,1},$$
and define the "right transformation":
$$\mathscr R_k: X_k \rightarrow \Phi_{d,1}, \varepsilon_{ab} \mapsto \varepsilon_{b+1,\varphi_k (a)}$$
for $k \in [2,2r-2]$.
Put
$$\mathcal{X}= \bigsqcup_{k \in [2,2r-2]} \mathcal{X}_k, \quad X = \bigsqcup_{k \in [2,2r-2]} X_k.$$
Thus we can define
$\varphi: \mathcal{X} \rightarrow  [1, n]$ sending $a\in \mathcal{X}_k$ to $\varphi_k (a)$ and a map
$\mathscr R: X \rightarrow  \Phi_{d,1} ,$
sending $\varepsilon_{ab} \in X_k$ to $\mathscr R_k(\varepsilon_{ab})$
for $k \in [2,2r-2]$.

Similarly, we put
$$\mathcal{Y}_k = \{ b\in [1,n] \ | \  j_{k}<b \leq j_{k+1} \ \text{and}\  b < j_k+j_{k-1}-j_{k-2}   \}, $$
for $k \in [2,2r-2]$, and let $\psi_k (b)= j_k+j_{k-1}-b $ for $b\in \mathcal{Y}_k$.  We set
$$Y_k = \{ \varepsilon_{ab} | j_{k-1}<a \leq j_k \leq b <j_{k+1},
\ b \in \mathcal{Y}_k  \} \subseteq \Phi_{d,1}$$
and define the "left transformation":
$$\mathscr L_k: Y_k \rightarrow \Phi_{d,1}, \varepsilon_{ab} \mapsto \varepsilon_{\psi_k (b),
a-1},$$ for $k \in [2,2r-2]$.
Put
$$\mathcal{Y}= \bigsqcup_{k \in [2,2r-2]} \mathcal{Y}_k, \quad Y = \bigsqcup_{k \in [2,2r-2]} Y_k.$$
Thus we can define
$\psi: \mathcal{Y} \rightarrow  [1, n]$ sending $b\in \mathcal{X}_k$ to $\psi_k (b)$ and a map
$\mathscr L: Y \rightarrow  \Phi_{d,1} ,$
sending $\varepsilon_{ab} \in Y_k$ to $\mathscr L_k(\varepsilon_{ab})$.

\smallskip
For $\varepsilon_{ab}  \in  \Phi_{d,1}$, if $b< i'_1= j_{r+1}$, then $\varepsilon_{ab}  \in X$. On the other hand for $\varepsilon_{ab}  \in  \Phi_{d,1}$,  if $a> i_1=j_{r}$, then $\varepsilon_{ab}  \in Y$. So we see that $X\cup Y =  \Phi_{d,1}$. Let $\tau: [1,n] \rightarrow [1,n]$, $a \mapsto n+1-a$ be an involution on $[1,n]$. By symmetry, there exists a $1-1$ correspondence between $X$ and $Y$, corresponding $\varepsilon_{ab} \in X$ to $\varepsilon_{\tau(b)\tau(a)} \in Y$.

Let $\varepsilon_{ab} \in X \backslash Y$  and we define  the $\mathscr{R}$-orbit of $\varepsilon_{ab}$
$$X_{ab} = \{ \varepsilon_{ab}, \mathscr R(\varepsilon_{ab}),
\cdots, \mathscr R^{m-1}(\varepsilon_{ab}) \},$$
where $m \in \mathbb N$ is maximal with $\mathscr R^{m-1}(\varepsilon_{ab}) \in X$. It  can be written as
$$\varepsilon_{a, b} \stackrel{\mathscr{R}}{\longrightarrow}  \varepsilon_{b+1, \varphi(a)}
\stackrel{\mathscr{R}}{\longrightarrow}  \varepsilon_{\varphi(a)+1, \varphi(b)-1}
\stackrel{\mathscr{R}}{\longrightarrow}  \varepsilon_{\varphi(b) , \varphi^2(a)-1}
\stackrel{\mathscr{R}}{\longrightarrow} \varepsilon_{ \varphi^2(a),  \varphi^2(b)}
\stackrel{\mathscr{R}}{\longrightarrow} \dots $$
Similarly for $\varepsilon_{ab} \in Y \backslash X$, we also define  the $\mathscr{L}$-orbit of $\varepsilon_{ab}$
$$Y_{ab} = \{ \varepsilon_{ab}, \mathscr L(\varepsilon_{ab}),
\cdots, \mathscr L^{l-1}(\varepsilon_{ab}) \},$$
where $l \in \mathbb N$ is maximal with $\mathscr R^{l-1}(\varepsilon_{ab}) \in Y$. It  can be written as
$$\varepsilon_{a, b} \stackrel{\mathscr{L}}{\longrightarrow}  \varepsilon_{\psi (b), a-1}
\stackrel{\mathscr{L}}{\longrightarrow}  \varepsilon_{\psi (a)+1, \psi (b)-1}
\stackrel{\mathscr{L}}{\longrightarrow}  \varepsilon_{\psi^2 (b)+1, \psi (a)}
\stackrel{\mathscr{L}}{\longrightarrow} \varepsilon_{\psi^2 (a), \psi^2 (b)}
\stackrel{\mathscr{L}}{\longrightarrow} \dots $$

By symmetry we see that $X_{ab} = Y_{\tau(b)\tau(a)}$ or $X_{ab} \cap Y_{\tau(b)\tau(a)} = \emptyset$.
For convenience, we set
\begin{eqnarray}
\Omega^0 & = & \{(a, b) \in [1, n] \times [1,n]\mid \varepsilon_{ab} \in X \backslash Y, \  X_{ab} = Y_{\tau(b)\tau(a)} \}, \nonumber \\
\Omega^1 & = & \{(a, b) \in [1, n] \times [1,n]\mid \varepsilon_{ab} \in X \backslash Y, \  X_{ab} \cap Y_{\tau(b)\tau(a)} = \emptyset \ \text{and} \ |X_{ab}| \ \text{is odd} \ \}, \nonumber \\
\Omega^2 & = & \{(a, b) \in [1, n] \times [1,n]\mid \varepsilon_{ab} \in X \backslash Y, \  X_{ab} \cap Y_{\tau(b)\tau(a)} = \emptyset\ \text{and} \ |X_{ab}| \ \text{is even} \ \}. \nonumber
\end{eqnarray}
We put $M_{ab} = X_{ab} \cup Y_{\tau(b)\tau(a)}$ and set $\Omega= \Omega^0 \cup \Omega^1 \cup \Omega^2$.
It is noticed that $\Phi_{d,1} = X \cup Y$. Then we can divide $\Phi_{d,1}$ as
$$\Phi_{d,1} = \bigsqcup_{(a, b) \in \Omega} M_{ab} $$

\begin{lemma} \label{link}
Let $(a, b)\in \Omega^1$, which implies that $|X_{ab}|=m$ is odd. Then there exists a unique integer $s$ such that
$$\mathscr{R}^{m-1}(\varepsilon_{ab})+ \mathscr{L}^s(\varepsilon_{\tau(b)\tau(a)})\in \Phi_{d,2}.$$
Moreover such  integer $s$  is  even.
\end{lemma}
\begin{proof}
Denote $\varepsilon_{a_k, b_k}=\mathscr{R}^{k-1}(\varepsilon_{a, b})$ for $k=1,2,\dots,m$.
By symmetry it is easy to see that
$$\varepsilon_{\tau(b_k), \tau(a_k)}=\mathscr{L}^{k-1}(\varepsilon_{\tau(b)\tau(a)}).$$
 Note that there exists an integer $l$ such that $$i'_{l-1}< a_m< i'_l \leq b_m  < i'_{l+1}.$$
 So we have $$j_{r+l-1}< a_m< j_{r+l} \leq b_m  < j_{r+l+1}.$$
 We will show that $l$ is an odd integer. Now suppose that $l$ is even,  then $\tau(a_m)= \psi^x(a_m)$ for some odd integer $x$. We consider the $\mathscr{L}$-orbit of $\varepsilon_{a_m, b_m}$ which is
$$\varepsilon_{a_m, b_m} \stackrel{\mathscr{L}}{\longrightarrow}  \varepsilon_{\psi (b_m), a_m-1}
\stackrel{\mathscr{L}}{\longrightarrow}  \varepsilon_{\psi (a_m)+1, \psi (b_m)-1}
\stackrel{\mathscr{L}}{\longrightarrow}  \varepsilon_{\psi^2 (b_m)+1, \psi (a_m)}
\stackrel{\mathscr{L}}{\longrightarrow} \varepsilon_{\psi^2 (a_m), \psi^2 (b_m)}
\stackrel{\mathscr{L}}{\longrightarrow} \dots. $$
Then $\varepsilon_{\psi^{x+1} (b_m)+1, \psi^{x} (a_m)}$ must be in this orbit.
Since $\varepsilon_{a_m, b_m}$ is not in $X$, $a_m$ is not in $\mathcal{X}_{r+l}$.
So $\tau(a_m)=\psi^x(a_m)$ is not in $Y_{l}$ .
 Therefore we have $$\varepsilon_{\psi^{x+1} (b_m)+1, \psi^{x} (a_m)}= \varepsilon_{a,b},$$
which implies that the  cardinality of the $\mathscr{R}$-orbit of $\varepsilon_{ab}$ is even. Thus we get a contradiction.
So the integer $l$ is odd and $\tau(b_m)= \psi^y(b_m)$ for some odd integer $y$.
It is not difficult to see that $$\varepsilon_{\varphi^y(\tau(b_m)), \varphi^{y+1}(\tau(a_m))-1 }= \varepsilon_{b_m, \varphi^{y+1}(\tau(a_m))-1 }$$ is in the $\mathscr{R}$-orbit of $\varepsilon_{\tau(b_m),\tau(a_m)}$. Note that the $\mathscr{R}$-orbit of $\varepsilon_{\tau(b_m),\tau(a_m)}$ is the same as $\mathscr{L}$-orbit of $\varepsilon_{\tau(b),\tau(a)}$.
 Therefore there exists a unique integer $s$ such that
 $$\mathscr{L}^s(\varepsilon_{\tau(b)\tau(a)}) =\varepsilon_{b_m, \varphi^{y+1}(\tau(a_m))-1 }.$$
So for such integer $s$ we have
$$\mathscr{R}^{m-1}(\varepsilon_{ab})+ \mathscr{L}^s(\varepsilon_{\tau(b)\tau(a)})\in \Phi_{d,2}.$$ Moreover it is not difficult to see that $m-1-s=2y$ which is even. 
As $m$ is odd, $s$ must be even.
\end{proof}

For $(a, b) \in \Omega^1$, we denote
$$\mathscr{R}^{m-1}(\varepsilon_{ab})+ \mathscr{L}^s(\varepsilon_{\tau(b)\tau(a)})\in \Phi_{d,2}$$
in Lemma \ref{link} as $\varpi_{ab}$. 
It is noticed that $\varpi_{ab}$ and $\varpi_{a'b'}$ may be the same for $(a,b) \neq (a',b')$ in $\Omega^1$.

\setcounter{subsection}{3}
\subsection{Proof of Theorem \ref{mainThm2} for type $A$}

Now we give the proof of our main theorem for type $A$.
We put
\begin{eqnarray}
\Lambda_{ab} =
\begin{cases}
\{ \mathscr R^{k-1}(\varepsilon_{ab}) + \mathscr R^k(\varepsilon_{ab}) |
k \in [1, |X_{ab}|-1] \} & \text{ if } (a,b) \in \Omega^0, \\
\{ \mathscr R^{k-1}(\varepsilon_{ab}) + \mathscr R^k(\varepsilon_{ab}),
 \mathscr L^{k}(\varepsilon_{\tau(b)\tau(a)}) + \mathscr L^{k-1}(\varepsilon_{\tau(b)\tau(a)}) |
k \in [1, |X_{ab}|-1] \} \bigcup \{ \varpi_{ab} \} & \text{ if } (a,b) \in \Omega^1, \\
\{ \mathscr R^{k-1}(\varepsilon_{ab}) + \mathscr R^k(\varepsilon_{ab}),
 \mathscr L^{k}(\varepsilon_{\tau(b)\tau(a)}) + \mathscr L^{k-1}(\varepsilon_{\tau(b)\tau(a)}) |
k \in [1, |X_{ab}|-1] \}   & \text{ if } (a,b) \in \Omega^2,
\end{cases} \nonumber
\end{eqnarray}
and define $\lambda \in \operatorname{Hom}_{\mathbb Z}(\mathfrak g_{\mathbb Z,d}(2), \mathbb Z)$ by
$$\lambda(e_{\alpha}) =
\begin{cases}
1 & \text{ if }\alpha \in \mathop{\bigcup}\limits_{(a,b) \in \Omega} \Lambda_{ab}, \\
0 & \text{ if }\alpha \in \Phi_{d,2} \backslash \mathop{\bigcup}\limits_{(a,b) \in \Omega} \Lambda_{ab}.
\end{cases}$$
Then
$$\det \mathscr G_{d,\lambda} = \pm \prod_{(a,b) \in \Omega} \det \mathscr G_{ab},$$
where $\mathscr G_{ab} = (\lambda([e_{\alpha},e_{\beta}]))_{\alpha,\beta \in M_{ab}}$ for each $(a,b) \in \Omega$.

\medskip
When $(a,b) \in \Omega^0$, 
the cardinality $|X_{ab}|$ is even
by symmetry. Then it is easy to notice that $\det \mathscr G_{ab}\in \{\pm 1\}$. 
When $(a,b) \in \Omega^1$, we write down the matrix $\mathscr G_{ab}$ and do cofactor expansion of the
determinant $\det \mathscr G_{ab}$. Using the result  of Lemma \ref{link}, it is checked that $\det \mathscr G_{ab}\in \{\pm 1\}$.
When $(a,b) \in \Omega^2$, there are two cases.
In the first case, the only subdiagonal of the matrix $\mathscr G_{ab}$ is $\pm 1$,
then  the form of this matric $\mathscr G_{ab}$ is the same as  $\mathscr G_{cd}$
for $(c,d) \in \Omega^0$. In the second case, the form of the matrix
$\mathscr G_{ab}$ is the same as  $\mathscr G_{cd}$ for $(c,d) \in \Omega^1$.
In both cases, we know that $\det \mathscr G_{ab}\in \{\pm 1\}$.
Therefore the map $\lambda: {\mathfrak g}_{\mathbb{Z},d}(2)\rightarrow \mathbb{Z}$
we consider here satisfies that $\sigma_{\lambda}$ is non-degenerate over $\mathbb{Z}$.

\medskip
Therefore we have proved Theorem \ref{mainThm2} (1). 
As every weighted Dynkin diagram is special in the type $A$ case, 
Theorem \ref{mainThm2} has been proved in 
this case.

\section{Faithful map} \label{s_faithful_map}

\subsection{Odd weights of types $C$, $B$ and $D$}
\setcounter{theorem}{1}
Let  $\mathfrak g$ be the simple Lie algebra of type $C_n$ over $\mathbf k$. The weighted Dynkin diagram associated to the pairs of partition $(\mu, \nu)$ with $|\mu|+|\nu|=n$ where $\nu$ has distinct parts  are given in Section \ref{WDD}.
We can identify $\Delta^o$ with the set of sequences $(i_1, i_2,\cdots, i_k)$ of integers, satisfying the following conditions (1)-(4):
\begin{enumerate}
\item $1 \leq i_1<i_2< \cdots <i_k \leq n-1$,
\item $s_l \leq s_{l+2}$ for $l \in [1,k-2]$,
\item $s_l \equiv 0$ for $l \in [1,k]$ with $l \equiv k-1$,
\item $s_{k-1} \leq 2(n-i_k)$,
\end{enumerate}
where $s_1 =i_1$ and $s_l = i_l-i_{l-1}$ for $l \in [2,k]$,
through identifying $d \in \Delta^o$ with $(i_1,i_2,\cdots,i_k)$ where $\{i_l | l \in [1,k]\} = \{ i \in [1,n] | d(\alpha_i) =1\}$.

\medskip

Let $\mathfrak g$ be the simple Lie algebra of type $B_n$ over $\mathbf k$.
The set $\Delta$ of weighted Dynkin diagrams 
is given in Section \ref{WDD}.
It is clear that
we can identify $\Delta^o$ with the set of sequences $(i_1, i_2,\cdots, i_k)$ of integers, satisfying the following (1)-(4):
\begin{enumerate}
\item $1 \leq i_1<i_2< \cdots <i_k \leq n$,
\item $s_l \leq s_{l+2}$ for $l \in [1,k-2]$,
\item $s_l \equiv 0$ for $l \in [1,k]$ with $l \equiv k$,
\item $s_{k-1} \leq 2(n-i_k)+1$,
\end{enumerate}
where $s_1 =i_1$ and $s_l = i_l-i_{l-1}$ for $l \in [2,k]$,
through identifying $d \in \Delta^o$ with $(i_1,i_2,\cdots,i_k)$ where $\{i_l| l \in [1,k]\} = \{ i \in [1,n]| d(\alpha_i) =1\}$.

\medskip

Let  $\mathfrak g$ be the simple Lie algebra of type $D_n$ over $\mathbf k$.
The set $\Delta$ of weighted Dynkin diagrams is given in Section \ref{WDD}.
If $\mathfrak d$ is a sequence of elementary divisors defining odd weighted Dynkin diagrams, then the two Dynkin diagrams in (\ref{d_D1}) and (\ref{d_D2}) coincide.
Consequently, the weighted Dynkin diagrams of odd weights can be classified in the following two cases:
\begin{eqnarray} \label{d_D1_1}
\text{
\begin{picture}(40,10)(130,20)
\put(50,0){\circle{6}}
\put(90,0){\circle{6}}
\put(130,0){\circle{6}}
\put(135,0){$\dots$}
\put(150,0){$\dots$}
\put(165,0){$\dots$}
\put(180,0){$\dots$}
\put(200,0){\circle{6}}
\put(231,27){\circle{6}}
\put(231,-27){\circle{6}}
\put(36,-15){$\xi_1-\xi_2$}
\put(77,-15){$\xi_2-\xi_3$}
\put(160,-15){$\xi_{n-2}-\xi_{n-1}$}
\put(240,23){$1$}
\put(240,-29){$1$}
\put(53,0){\line(1,0){34}}
\put(93,0){\line(1,0){34}}
\put(203,1){\line(1,1){25}}
\put(203,-1){\line(1,-1){25}}
\end{picture}}
\end{eqnarray}
\vspace{20 mm}

\noindent and
\begin{eqnarray} \label{d_D2_1}
\text{
\begin{picture}(40,10)(130,20)
\put(50,0){\circle{6}}
\put(90,0){\circle{6}}
\put(130,0){\circle{6}}
\put(135,0){$\dots$}
\put(150,0){$\dots$}
\put(165,0){$\dots$}
\put(180,0){$\dots$}
\put(200,0){\circle{6}}
\put(231,27){\circle{6}}
\put(231,-27){\circle{6}}
\put(36,-15){$\xi_1-\xi_2$}
\put(77,-15){$\xi_2-\xi_3$}
\put(160,-15){$\xi_{n-2}-\xi_{n-1}$}
\put(240,23){$0$}
\put(240,-29){$0$}
\put(53,0){\line(1,0){34}}
\put(93,0){\line(1,0){34}}
\put(203,1){\line(1,1){25}}
\put(203,-1){\line(1,-1){25}}
\end{picture}}
\end{eqnarray}
\vspace{20 mm}

Thus we can write $\Delta^o = \Delta^o_1 \sqcup \Delta^o_2$, where $\Delta^o_1$ (resp. $\Delta^o_2$) consists of weighted Dynkin diagrams of the form (\ref{d_D1_1}) (resp. (\ref{d_D2_1})).

The weighted Dynkin diagram  associated to the partition $(2^{2m}1^2)$ is special, in this case it is not difficult to see that there exists a non-degenerate $\lambda \in \operatorname{Hom}_{\mathbb Z}( \mathfrak g_{\mathbb Z,d}(2),\mathbb Z)$. Thus we exclude this situation.

The set $\Delta^o_1$ can be identified with the set of sequences $(i_1, i_2,\cdots, i_k)$ of integers, satisfying the following conditions (1)-(4):
\begin{enumerate}
\item $1 \leq i_1<i_2< \cdots <i_k \leq n-3$,
\item $s_l \leq s_{l+2}$ for $l \in [1,k-1]$,
\item $s_l \equiv 0$ for $l \in [1,k+1]$ with $l \equiv k+1$,
\item $s_l\leq 2$ for $l \in [1,k]$ with $l \equiv k$.
\end{enumerate}
where $s_1 =i_1$ and $s_l = i_l-i_{l-1}$ for $l \in [2,k]$ and $s_{k+1}=n-1-i_k$,
through identifying 
$d \in \Delta^o_1$ with $(i_1,i_2,\cdots,i_k)$ where $\{i_l| l \in [1,k]\} = \{ i \in [1,n]| d(\alpha_i) =1\}$.

\medskip

The set $\Delta^o_2$ can be identified with the set of sequences $(i_1, i_2,\cdots, i_k)$ of integers, satisfying the following conditions (1)-(4):
\begin{enumerate}
\item $1 \leq i_1<i_2< \cdots <i_k \leq n-2$,
\item $s_l \leq s_{l+2}$ for $l \in [1,k-2]$,
\item $s_l \equiv 0$ for $l \in [1,k]$ with $l \equiv k$,
\item $s_{k-1} \leq 2(n-i_k)$.
\end{enumerate}
where $s_1 =i_1$ and $s_l = i_l-i_{l-1}$ for $l \in [2,k]$,
through identifying $d \in \Delta^o$ with $(i_1,i_2,\cdots,i_k)$ where $\{i_l| l \in [1,k]\} = \{ i \in [1,n]| d(\alpha_i) =1\}$.

\setcounter{theorem}{2}
\setcounter{equation}{0}
\subsection{Faithful maps} \label{ssfai}
For a fixed simple root system and an odd weighted Dynkin diagram $d =(i_1,i_2,\cdots,i_k) \in \Delta^o$, put
$$Y_{d,l} = \{ \alpha \in \Phi_{d,1} \mid [ \alpha : \alpha_{i_l}]=1 \},$$
for $l \in [1,k]$. Then $|Y_{d,l}| \leq | Y_{d,l+1}|$ for $l \in [1,k-1]$.
In this subsection, we fix $d =(i_1,i_2,\cdots, i_k) \in \Delta^o$, where $k \geq 3$, and $\lambda \in \operatorname{Hom}_{\mathbf k}(\mathfrak g_{d}(2), \mathbf k)$, or
$\lambda \in \operatorname{Hom}_{\mathbb Z}(\mathfrak g_{\mathbb Z, d}(2), \mathbb Z)$.
Set $s_1 = i_1$, and $s_j = i_j - i_{j-1}$ for $j \in [2,k]$.


Now we set
\begin{eqnarray}
\Omega_l = \{ \varepsilon_s- \varepsilon_t \in \Phi_{d,2} \mid t-s \leq i_{l+1} -(i_{l-1}+1)\},
\end{eqnarray}
for $l \in [1,k-2]$. If
$$\{ \gamma \in \Phi_{d,2} | [\gamma : \alpha_{i_l}]= [\gamma: \alpha_{i_{l+1}}]=1, \lambda(e_{\gamma}) \neq 0 \} \subseteq \Omega_l,$$
 for $l \in [1,k-2]$, then $\lambda$ is said to be \emph{faithful}.

If $\lambda$ is faithful, then put $Q_1 = Y_{d,1}$, and
\begin{eqnarray}
P_l & = & \{ \alpha \in Y_{d,l+1} \mid \text{there exists } \beta \in Q_l, \text{ such that } \beta + \alpha \in \Omega_l \}, \nonumber \\
Q_{l+1} & = & Y_{d,l+1} \backslash P_l, \nonumber \\
M_{l+1} & = & (\lambda([e_{\alpha},e_{\beta}]))_{\alpha \in Q_l,\beta \in P_l}, \nonumber
\end{eqnarray}
for $l \in [1,k-2]$.
Put
\begin{eqnarray} \label{e_MK}
M_k = (\lambda([e_{\alpha},e_{\beta}]))_{\alpha, \beta \in Q_{k-1} \cup Y_{d,k}}.
\end{eqnarray}

It is convenient to describe the sets $Q_1,P_1,\cdots, Q_{k-2},P_{k-2},Q_{k-1}$ through diagrams. For example, the set $Q_1 = Y_{d,1} = \{ \varepsilon_s- \varepsilon_t \mid s \in [1,i_1], t\in [i_1,i_2-1]\}$ is denoted as


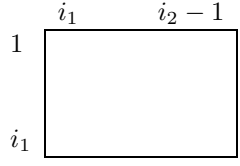
\begin{center}

\begin{tabular}{lccccc}
    & $i_1$ & & & $i_2-1$ \\
    \cline{2-6}
$1$ & \multicolumn{5}{|c|}{\multirow{4}{*}{ }}\\
& \multicolumn{5}{|c|}{\multirow{4}{*}{ }}\\
& \multicolumn{5}{|c|}{\multirow{4}{*}{ }}\\
$i_1$ & \multicolumn{5}{|c|}{\multirow{4}{*}{ }}\\
\cline{2-6}
\end{tabular}
\end{center}

\noindent Then $Y_{d,2} = P_1 \sqcup Q_2$ is denoted as follows,

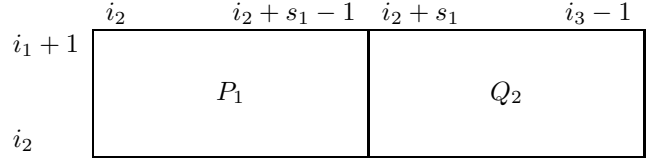
\begin{center}
\begin{tabular}{lcccccccccc}
    & $i_2$ & & & & $i_2+s_1-1$ & $i_2+s_1$ &  & & & $i_3-1$\\
    \cline{2-11}
$i_1+1$ & \multicolumn{5}{|c|}{\multirow{4}{*}{ $P_1$ }} & \multicolumn{5}{c|}{\multirow{4}{*}{ $Q_2$ }}\\
& \multicolumn{5}{|c|}{\multirow{4}{*}{ }} & \multicolumn{5}{c|}{\multirow{4}{*}{ }}\\
& \multicolumn{5}{|c|}{\multirow{4}{*}{ }} & \multicolumn{5}{c|}{\multirow{4}{*}{ }}\\
$i_2$ & \multicolumn{5}{|c|}{\multirow{4}{*}{ }} & \multicolumn{5}{c|}{\multirow{4}{*}{ }}\\
\cline{2-11}
\end{tabular}
\end{center}

\noindent  from which we can see that
\begin{eqnarray}
P_1 & = & \{ \varepsilon_s- \varepsilon_t \mid s \in [i_1+1,i_2], t \in [s_2,i_2+s_1-1]\}, \nonumber \\
Q_2 & = & \{ \varepsilon_s- \varepsilon_t \mid s \in [i_1+1,i_2], t \in [i_2+s_1,i_3-1]\}. \nonumber
\end{eqnarray}

In general, for $l \in [2,k-1]$, the partition of $Y_{d,l}$ into $P_{l-1} \sqcup Q_l$ is

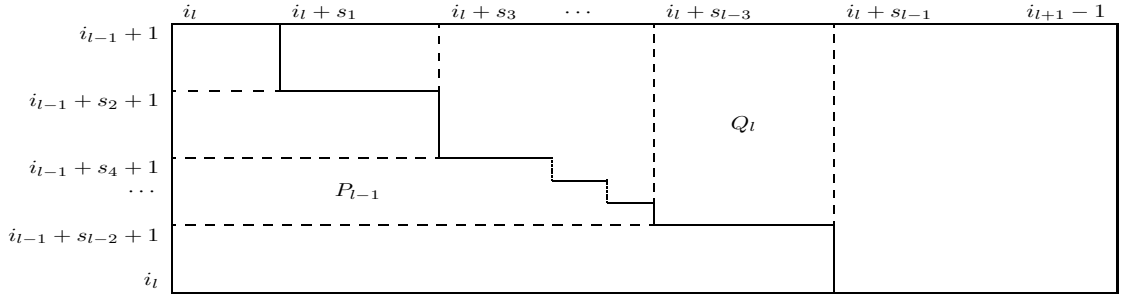
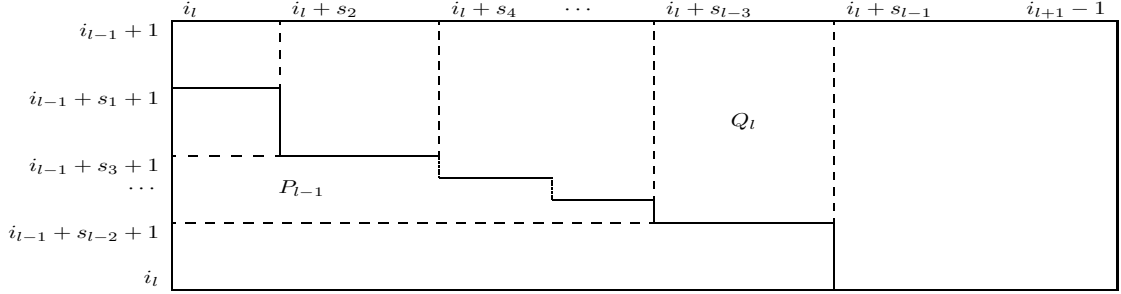
\begin{center}
\resizebox{\textwidth}{20mm}{
\begin{tabular}{r cccc cccc ccccc cccc ccccc}
    & $i_l$ & & & &
    $i_l+s_1$  &  & & &
    $i_l+s_3$ &&  $\cdots$ &&&
    $i_l+s_{l-3}$ & &&&
    $i_l+s_{l-1}$ &&&& $i_{l+1}-1$ \\
    \cline{2-23}
$i_{l-1}+1$ & \multicolumn{4}{|c|}{\multirow{3}{*}{   }} &  \multicolumn{4}{c;{4pt/4pt}}{\multirow{3}{*}{  }} &
\multicolumn{5}{c;{4pt/4pt}}{\multirow{8}{*}{  }} &
\multicolumn{4}{c;{4pt/4pt}}{\multirow{9}{*}{ $Q_l$   }} &
\multicolumn{5}{c|}{\multirow{12}{*}{  }}
\\
& \multicolumn{4}{|c|}{\multirow{3}{*}{ }} &
\multicolumn{4}{c;{4pt/4pt}}{\multirow{3}{*}{  }} &
\multicolumn{5}{c;{4pt/4pt}}{\multirow{8}{*}{  }} &
\multicolumn{4}{c;{4pt/4pt}}{\multirow{9}{*}{  }} &
\multicolumn{5}{c|}{\multirow{12}{*}{  }}
\\
& \multicolumn{4}{|c|}{\multirow{3}{*}{ }} &
\multicolumn{4}{c;{4pt/4pt}}{\multirow{3}{*}{  }} &
\multicolumn{5}{c;{4pt/4pt}}{\multirow{8}{*}{  }} &
\multicolumn{4}{c;{4pt/4pt}}{\multirow{9}{*}{  }} &
\multicolumn{5}{c|}{\multirow{12}{*}{  }}
\\
\cdashline{2-5}\cline{6-9}
$i_{l-1}+s_2+1$ & \multicolumn{8}{|c|}{\multirow{3}{*}{ }} & \multicolumn{5}{c;{4pt/4pt}}{\multirow{8}{*}{ }} &
\multicolumn{4}{c;{4pt/4pt}}{\multirow{9}{*}{ }} &
\multicolumn{5}{c|}{\multirow{12}{*}{ }}
\\
& \multicolumn{8}{|c|}{\multirow{3}{*}{ }} &
\multicolumn{5}{c;{4pt/4pt}}{\multirow{8}{*}{ }} &
\multicolumn{4}{c;{4pt/4pt}}{\multirow{9}{*}{ }} &
\multicolumn{5}{c|}{\multirow{12}{*}{ }}
\\
& \multicolumn{8}{|c|}{\multirow{3}{*}{ }} &
\multicolumn{5}{c;{4pt/4pt}}{\multirow{8}{*}{ }} &
\multicolumn{4}{c;{4pt/4pt}}{\multirow{9}{*}{ }} &
\multicolumn{5}{c|}{\multirow{12}{*}{ }}
\\
\cdashline{2-9} \cline{10-11}
$i_{l-1}+s_4+1$ & \multicolumn{10}{|c;{1pt/1pt}}{\multirow{3}{*}{$P_{l-1}$ }} &
\multicolumn{3}{c;{4pt/4pt}}{\multirow{8}{*}{ }} &
\multicolumn{4}{c;{4pt/4pt}}{\multirow{9}{*}{ }} &
\multicolumn{5}{c|}{\multirow{12}{*}{ }}
\\
\cline{12-12}
$\cdots$  & \multicolumn{11}{|c;{1pt/1pt}}{\multirow{3}{*}{ }} &
\multicolumn{2}{c;{4pt/4pt}}{\multirow{8}{*}{ }} &
\multicolumn{4}{c;{4pt/4pt}}{\multirow{9}{*}{ }} &
\multicolumn{5}{c|}{\multirow{12}{*}{ }}
\\
\cline{13-14}
 & \multicolumn{13}{|c|}{\multirow{3}{*}{ }} &
\multicolumn{4}{c;{4pt/4pt}}{\multirow{9}{*}{ }} &
\multicolumn{5}{c|}{\multirow{12}{*}{ }}
\\
\cdashline{2-14} \cline{15-18}
$i_{l-1}+s_{l-2}+1$ & \multicolumn{17}{|c|}{\multirow{3}{*}{ }} &
\multicolumn{5}{c|}{\multirow{12}{*}{ }}
\\
& \multicolumn{17}{|c|}{\multirow{3}{*}{ }} &
\multicolumn{5}{c|}{\multirow{12}{*}{ }}
\\
$i_l$ & \multicolumn{17}{|c|}{\multirow{3}{*}{ }} &
\multicolumn{5}{c|}{\multirow{12}{*}{ }}
\\
\cline{2-23}
\end{tabular}
}
\end{center}

if $l \equiv 0$, or

\begin{center}
\resizebox{\textwidth}{20mm}{
\begin{tabular}{r cccc cccc ccccc cccc ccccc}
    & $i_l$ & & & &
    $i_l+s_2$  &  & & &
    $i_l+s_4$ &&  $\cdots$ &&&
    $i_l+s_{l-3}$ & &&&
    $i_l+s_{l-1}$ &&&& $i_{l+1}-1$ \\
    \cline{2-23}
$i_{l-1}+1$ & \multicolumn{4}{|c;{4pt/4pt}}{\multirow{3}{*}{   }} &  \multicolumn{4}{c;{4pt/4pt}}{\multirow{3}{*}{  }} &
\multicolumn{5}{c;{4pt/4pt}}{\multirow{8}{*}{  }} &
\multicolumn{4}{c;{4pt/4pt}}{\multirow{9}{*}{ $Q_l$   }} &
\multicolumn{5}{c|}{\multirow{12}{*}{  }}
\\
& \multicolumn{4}{|c;{4pt/4pt}}{\multirow{3}{*}{ }} &
\multicolumn{4}{c;{4pt/4pt}}{\multirow{3}{*}{  }} &
\multicolumn{5}{c;{4pt/4pt}}{\multirow{8}{*}{  }} &
\multicolumn{4}{c;{4pt/4pt}}{\multirow{9}{*}{  }} &
\multicolumn{5}{c|}{\multirow{12}{*}{  }}
\\
& \multicolumn{4}{|c;{4pt/4pt}}{\multirow{3}{*}{ }} &
\multicolumn{4}{c;{4pt/4pt}}{\multirow{3}{*}{  }} &
\multicolumn{5}{c;{4pt/4pt}}{\multirow{8}{*}{  }} &
\multicolumn{4}{c;{4pt/4pt}}{\multirow{9}{*}{  }} &
\multicolumn{5}{c|}{\multirow{12}{*}{  }}
\\
\cline{2-5}
$i_{l-1}+s_1+1$ & \multicolumn{4}{|c|}{\multirow{3}{*}{ }} &
\multicolumn{4}{c;{4pt/4pt}}{\multirow{3}{*}{ }} &
\multicolumn{5}{c;{4pt/4pt}}{\multirow{8}{*}{ }} &
\multicolumn{4}{c;{4pt/4pt}}{\multirow{9}{*}{ }} &
\multicolumn{5}{c|}{\multirow{12}{*}{ }}
\\
& \multicolumn{4}{|c|}{\multirow{3}{*}{ }} &
\multicolumn{4}{c;{4pt/4pt}}{\multirow{3}{*}{ }} &
\multicolumn{5}{c;{4pt/4pt}}{\multirow{8}{*}{ }} &
\multicolumn{4}{c;{4pt/4pt}}{\multirow{9}{*}{ }} &
\multicolumn{5}{c|}{\multirow{12}{*}{ }}
\\
& \multicolumn{4}{|c|}{\multirow{3}{*}{ }} &
\multicolumn{4}{c;{4pt/4pt}}{\multirow{3}{*}{ }} &
\multicolumn{5}{c;{4pt/4pt}}{\multirow{8}{*}{ }} &
\multicolumn{4}{c;{4pt/4pt}}{\multirow{9}{*}{ }} &
\multicolumn{5}{c|}{\multirow{12}{*}{ }}
\\
\cdashline{2-5} \cline{6-9}
$i_{l-1}+s_3+1$ & \multicolumn{8}{|c;{1pt/1pt}}{\multirow{3}{*}{$P_{l-1}$ }} &
\multicolumn{5}{c;{4pt/4pt}}{\multirow{8}{*}{ }} &
\multicolumn{4}{c;{4pt/4pt}}{\multirow{9}{*}{ }} &
\multicolumn{5}{c|}{\multirow{12}{*}{ }}
\\
\cline{10-11}
$\cdots$  & \multicolumn{10}{|c;{1pt/1pt}}{\multirow{3}{*}{ }} &
\multicolumn{3}{c;{4pt/4pt}}{\multirow{8}{*}{ }} &
\multicolumn{4}{c;{4pt/4pt}}{\multirow{9}{*}{ }} &
\multicolumn{5}{c|}{\multirow{12}{*}{ }}
\\
\cline{12-14}
 & \multicolumn{13}{|c|}{\multirow{3}{*}{ }} &
\multicolumn{4}{c;{4pt/4pt}}{\multirow{9}{*}{ }} &
\multicolumn{5}{c|}{\multirow{12}{*}{ }}
\\
\cdashline{2-14} \cline{15-18}
$i_{l-1}+s_{l-2}+1$ & \multicolumn{17}{|c|}{\multirow{3}{*}{ }} &
\multicolumn{5}{c|}{\multirow{12}{*}{ }}
\\
& \multicolumn{17}{|c|}{\multirow{3}{*}{ }} &
\multicolumn{5}{c|}{\multirow{12}{*}{ }}
\\
$i_l$ & \multicolumn{17}{|c|}{\multirow{3}{*}{ }} &
\multicolumn{5}{c|}{\multirow{12}{*}{ }}
\\
\cline{2-23}
\end{tabular}
}
\end{center}


if $l \not\equiv 0$. Thus it is noticed that
\begin{eqnarray}
|P_l| = |Q_l|  =
\begin{cases}
\sum\limits_{j =1}^{l/2} (s_{2j+1}-s_{2j-1})s_{2j} & \text{ if } l \equiv 0, \\
\sum\limits_{j=1}^{(l+1)/2} (s_{2j} - s_{2j-2})s_{2j-1} & \text{ if } l \not\equiv 0,
\end{cases} \nonumber
\end{eqnarray}
for $l \in [1,k-2]$.
Consequently,
\begin{eqnarray} \label{e_prod}
\det \mathscr G_{d,\lambda} = \pm (\prod_{l=2}^{k-1} \det M_l)^2 \det M_k
\end{eqnarray}

\setcounter{theorem}{3}
\subsection{Isomorphisms} \label{isomorphism}
Let $d =(i_1,i_2,\cdots,i_k) \in \Delta^o$, and $\lambda, \lambda' \in \operatorname{Hom}_{\mathbf k}(\mathfrak g_{d}(2), \mathbf k)$.
For $\Psi: \mathfrak g \rightarrow \mathfrak g$ an  isomorphism of Lie algebras,
if $\Psi(\mathfrak g_{d}(1))=  \mathfrak g_{d}(1)$, and
$$\lambda'([\Psi(v),\Psi(w)]) = \lambda(v,w),$$
for all $v,w \in \mathfrak g_d(1)$, then we call
$\Psi: (\mathfrak g, d, \lambda) \rightarrow (\mathfrak g, d, \lambda')$
an \emph{isomorphism of weighted Lie algebras}.

There is a similar definition in the integral case. Let $\lambda, \lambda' \in \operatorname{Hom}_{\mathbb Z}(\mathfrak g_{\mathbb Z,d}(2), \mathbb Z)$.
For $\Psi: \mathfrak g_0 \rightarrow \mathfrak g_0$ an  isomorphism of Lie algebras,
if $\Psi(\mathfrak g_{\mathbb Z,d}(1))=  \mathfrak g_{\mathbb Z,d}(1)$, and
$$\lambda'([\Psi(v),\Psi(w)]) = \lambda(v,w),$$
for all $v,w \in \mathfrak g_{\mathbb Z,d}(1)$, then we call
$\Psi: (\mathfrak g_0, d, \lambda) \rightarrow (\mathfrak g_0, d, \lambda')$
an \emph{isomorphism of weighted Lie algebras}.

For $\lambda, \lambda' \in \operatorname{Hom}_{\mathbf k}(\mathfrak g_{d}(2), \mathbf k)$ (resp. $\lambda, \lambda' \in \operatorname{Hom}_{\mathbb Z}(\mathfrak g_{\mathbb Z,d}(2), \mathbb Z)$), it is clear that
\begin{eqnarray} \label{e_iso_eq}
\det \mathscr G_{d,\lambda} = \pm \det \mathscr G_{d,\lambda'}
\end{eqnarray}
if there exists an isomorphism $\Psi: (\mathfrak g, d, \lambda) \rightarrow (\mathfrak g, d', \lambda')$ (resp. $\Psi: (\mathfrak g_0, d, \lambda) \rightarrow (\mathfrak g_0, d', \lambda')$) of weighted Lie algebras.

Let $d =(i_1,i_2,\cdots, i_k) \in \Delta^o$, where 
$k \geq 2$. Let $\lambda \in \operatorname{Hom}_{\mathbf k}(\mathfrak g_{d}(2),\mathbf k)$.
For 
$l \in [1,k-1]$ and $s,t \in [i_l,i_{l+1}-1]$, there are natural isomorphisms as follows. In the following we denote $e_{s,t}=e_{\varepsilon_s- \varepsilon_t}$ for convenience.
\begin{enumerate}
\item For $\gamma \in \mathbf k$, define
$$ \varrho_{s,t}^{\gamma} : (\mathfrak g, d, \lambda) \rightarrow (\mathfrak g, d, \lambda'),$$
    where
    \begin{eqnarray}
    \begin{array}{rcll}
    \varrho_{s,t}^{\gamma} (e_{a,s}) & = & e_{a, s} + \gamma e_{a, t} & \text{ for all } a \in [i_{l-1}+1, i_l], \\
    \varrho_{s,t}^{\gamma} (e_{t+1, b}) & = & e_{t+1, b} - \gamma e_{s+1, b} & \text{ for all } b \in [i_{l+1},i_{l+2}-1] \text{ if } l \leq k-2, \\
    & & & \hspace{-0.44cm}\text{ or for all } b \in [i_k, 2n-i_k-1] \text{ if } l = k-1,  \\
    \varrho_{s,t}^{\gamma} (e_{\alpha}) & = & e_{\alpha}  & 
    \text{ for other } \alpha \in \Phi_{d,1},
    \end{array} \nonumber
    \end{eqnarray}
    and
    \begin{eqnarray}
    \begin{array}{rcll}
    \lambda'(e_{a', s}) & = &  \lambda(e_{a', s}) - \gamma \lambda(e_{a', t})
    & \text{ for all } a' \in [i_{l-2}+1,i_{l-1}] \text{ if } l \geq 2,  \\
    \lambda'(e_{t+1, b'}) & = & \lambda(e_{t+1, b'}) + \gamma \lambda(e_{s+1, b'}) & \text{ for all } b' \in [i_{l+2},i_{l+3}-1] \text{ if } l \leq k-3, \\
    & & &\hspace{-0.44cm} \text{ or for all } b' \in [i_k, 2n-i_k-1] \text{ if } l=k-2, \\
    \lambda'(e_{\beta}) & = & \lambda(e_{\beta})& \text{ for other } \beta \in \Phi_{d,2}.
    \end{array} \nonumber
    \end{eqnarray}
    Denote this $\lambda'$ as $\varrho_{s,t}^{\gamma}(\lambda)$.
    When $\gamma =1$, write $\varrho_{s,t}^{\gamma}$ as $\varrho_{s,t}$.

\item Define
$$ \varsigma_{s,t}: (\mathfrak g, d, \lambda) \rightarrow (\mathfrak g, d, \lambda'),$$
where
    \begin{eqnarray}
    \begin{array}{rcll}
    \varsigma_{s,t}(e_{a, s}) & = & e_{a, t} & \text{ for all } a \in [i_{l-1}+1, i_l], \\
    \varsigma_{s,t}(e_{t+1, b}) & = &  e_{s+1, b} &  \text{ for all } b \in [i_{l+1},i_{l+2}-1] \text{ if } l \leq k-2, \\
    & & & \hspace{-0.44cm}  \text{ or for all } b \in [i_k , 2 n-i_k - 1] \text{ if } l = k-1, \\
    \varsigma_{s,t}(e_{\alpha}) & = & e_{\alpha}  & 
    \text{ for other } \alpha \in \Phi_{d,1},
    \end{array} \nonumber
    \end{eqnarray}
    and
    \begin{eqnarray}
    \begin{array}{rcll}
    \lambda'(e_{a', s}) & = & \lambda(e_{a', t}) & \text{ for all } a' \in [i_{l-2}+1,i_{l-1}] \text{ if } l \geq 2, \\
    \lambda'(e_{t+1, b'}) & = & \lambda(e_{s+1, b'}) & \text{ for all } b' \in [i_{l+2},i_{l+3}-1] \text{ if } l \leq k-3,  \\
    & & & \hspace{-0.44cm} \text{ or for all } b' \in  [i_k, 2n-i_k-1] \text{ if } l=k-2, \\
    \lambda'(e_{\beta}) & = & \lambda(e_{\beta}) & \text{ for other } \beta \in \Phi_{d,2}.
    \end{array} \nonumber
    \end{eqnarray}
    Denote this $\lambda'$ as $\varsigma_{s,t}(\lambda)$.
\end{enumerate}

Similarly,
\begin{eqnarray}
\varrho_{s,t}  & : & (\mathfrak g_0, d, \lambda) \rightarrow (\mathfrak g_0, d, \varrho_{s,t}(\lambda)), \nonumber \\
\varsigma_{s,t} &: & (\mathfrak g_0, d, \lambda) \rightarrow (\mathfrak g_0, d, \varsigma_{s,t}(\lambda)), \nonumber
\end{eqnarray}
are defined for $\lambda \in \operatorname{Hom}_{\mathbb Z}(\mathfrak g_{\mathbb Z,d}(2),\mathbb Z)$.

\begin{lemma} \label{l_nnnw}
Let $d =(i_1,i_2,\cdots,i_k) \in \Delta^o$, and $l \in [2,k-1]$. Let $t_1,t_2 \in [i_l,i_{l+1}-1]$, $t_1 \neq t_2$, and $s \in [i_{l-2}+1,i_{l-1}]$.
For $\lambda \in \operatorname{Hom}_{\mathbb Z}(\mathfrak g_{\mathbb Z,d}(2),\mathbb Z)$, if $\lambda(e_{s,t_1}) \neq 0$, then
we can find a sequence of maps $$\varpi_1, \varpi_2, \cdots, \varpi_u$$
with $\varpi_i \in \{ \varrho_{t_1,t_2}^{\pm 1},  \varsigma_{t_1,t_2} \}$ for $i \in [1,u]$, and $\varpi_u \cdots \varpi_2 \varpi_1(\lambda)(e_{s,t_1}) =0$.
\end{lemma}

\begin{proof}
If $\lambda(e_{s,t_2}) =0$, then $\varsigma_{t_1,t_2}(\lambda)(e_{s,t_1})=0$, and the lemma is proved.
Now suppose $\lambda(e_{s,t_2}) \neq 0$. We show the result using induction on $|\lambda(e_{s,t_2})|$.

If $|\lambda(e_{s,t_2})| \leq |\lambda(e_{s,t_1})|$, then take $\varpi_1 \in \{ \varrho_{t_1,t_2}^{\pm 1} \}$, such that $$|\varpi_1(\lambda)(e_{s,t_1})| =
|\lambda(e_{s,t_1})| - | \lambda(e_{s,t_2})| < |\lambda(e_{s,t_1})|.$$ So the lemma follows by induction.

If $|\lambda(e_{s,t_2})| > | \lambda(e_{s,t_1})|$, then take $m \in \mathbb Z_+$, satisfying
$$m|\lambda(e_{s,t_1})|\leq |\lambda(e_{s,t_2})| <(m+1)|\lambda(e_{s,t_1})|.$$
Then we can take $\varpi_1, \cdots, \varpi_m \in \{ \varrho_{t_1,t_2}^{\pm 1} \}$, such that
$$|\varpi_m \cdots \varpi_1 \varsigma_{t_1,t_2}(\lambda)(e_{s,t_1})| = |\lambda(e_{s,t_2})| - m|\lambda(e_{s,t_1})| < |\lambda(e_{s,t_1})|. $$
The lemma follows by induction.
\end{proof}

\begin{lemma} \label{l_iso_faith}
Let $d =(i_1,i_2,\cdots, i_k) \in \Delta^o$, where $k \geq 3$.
\begin{enumerate}
\item Let $\lambda \in \operatorname{Hom}_{\mathbf k}(\mathfrak g_{d}(2),\mathbf k)$. There exists a faithful map $\lambda' \in \operatorname{Hom}_{\mathbf k}(\mathfrak g_{d}(2),\mathbf k)$, and an isomorphism $\Psi: (\mathfrak g, d,\lambda) \rightarrow (\mathfrak g, d, \lambda')$ of weighted Lie algebras.
\item Let $\lambda \in \operatorname{Hom}_{\mathbb Z}(\mathfrak g_{\mathbb Z,d}(2),\mathbb Z)$. There exists a faithful map $\lambda' \in \operatorname{Hom}_{\mathbb Z}(\mathfrak g_{\mathbb Z,d}(2),\mathbb Z)$, and an isomorphism $\Psi: (\mathfrak g_0, d,\lambda) \rightarrow (\mathfrak g_0, d, \lambda')$ of weighted Lie algebras.
\end{enumerate}
\end{lemma}

\begin{proof}
For $\lambda \in \operatorname{Hom}_{\mathbf k}(\mathfrak g_{d}(2),\mathbf k)$, or $\lambda \in \operatorname{Hom}_{\mathbb Z}(\mathfrak g_{\mathbb Z,d}(2),\mathbb Z)$,
put
$$\Gamma_l = \{ \gamma \in \Phi_{d,2} \mid [\gamma : \alpha_{i_l}]= [\gamma: \alpha_{i_{l+1}}]=1, \lambda(e_{\gamma}) \neq 0 \}$$
for $l \in [1,k-2]$, then
take $s$ minimal with $$\varepsilon_s -\varepsilon_t \in \mathop{\bigcup}\limits_{l=1}^{k-2}\left(\Gamma_l \backslash \Omega_l\right)$$
for some $t$. The result will be proved by induction on $s$.

Take $l \in [1,k-2]$ with $\varepsilon_s -\varepsilon_t  \in \Gamma_l \backslash \Omega_l$.
Let $t' = s+ i_{l+1} - (i_{l-1}+1)$. Then $t' <t$ as $\varepsilon_s -\varepsilon_t \not\in \Omega_l$.
\begin{enumerate}
\item Let $\lambda \in \operatorname{Hom}_{\mathbf k}(\mathfrak g_{d}(2),\mathbf k)$. If $\lambda(e_{st'}) =0$, then $\varsigma_{t,t'}(\lambda)(e_{st}) =0$, and $\varsigma_{t,t'}(\lambda)(e_{s_1t_1}) = \lambda(e_{s_1t_1})$ for any $s_1 \in [i_{l-1}+1,i_l]$ and $t_1 \in [i_{l+1},i_{l+2}-1]$ with $s_1 <s$.

If $\lambda(e_{st'}) \neq 0$, then $\varrho_{t,t'}^{\gamma}(\lambda)(e_{st}) =0$, where $\gamma=\lambda(e_{st})\lambda(e_{st'})^{-1}$. Moreover,
$\varrho_{t,t'}^{\gamma}(\lambda)(e_{s_1,t_1}) =  \lambda (e_{s_1,t_1})$, for any $s_1 \in [i_{l-1}+1,i_l]$ and $t_1 \in [i_{l+1},i_{l+2}-1]$ with $s_1 <s$.

In each case, the result will be proved through induction with $\lambda$ replaced by $\varsigma_{t,t'}(\lambda)$ or $\varrho_{t,t'}^{\gamma}(\lambda)$.
\item Let $\lambda \in \operatorname{Hom}_{\mathbb Z}(\mathfrak g_{\mathbb Z,d}(2),\mathbb Z)$. By Lemma \ref{l_nnnw}, there exist $\varpi_1, \cdots, \varpi_u \in \{ \varrho_{t,t'}^{\pm 1},  \varsigma_{t,t'} \}$, satisfying $\varpi_u \cdots \varpi_1(\lambda)(e_{s,t}) =0$. Moreover,
for any $s_1 \in [i_{l-1}+1,i_l]$ and $t_1 \in [i_{l+1},i_{l+2}-1]$ with $s_1 <s$,
$\varpi_u \cdots \varpi_1(\lambda)(e_{s_1,t_1}) =  \lambda (e_{s_1,t_1})$.

Replace $\lambda$ by $\varpi_u \cdots \varpi_1(\lambda)$, then the induction hypothesis is applied and the result is true.
\end{enumerate}

\end{proof}

In Theorem \ref{mainThm2}, for a given  map $\lambda:\mathfrak g_d(2)\rightarrow \mathbf k$ (resp. $\lambda: \mathfrak g_{\mathbb Z,d}(2) \rightarrow \mathbb Z$),
we try to figure out the properties of $\det \mathscr G_{d,\lambda}$.
The above arguments from subsection \ref{ssfai} to Lemma \ref{l_iso_faith} tell us that we 
can focus our attention to faithful $\lambda$.
According to subsection \ref{ssfai}, when $\lambda$ is faithful,
we can construct the set $P_l, Q_l$ and thus 
$$\det \mathscr G_{d,\lambda} = \pm (\prod_{l=2}^{k-1} \det M_l)^2 \det M_k.$$
The determinant $ \det M_l$, for $l\in [2, k-1]$, is easy to deal with. So we will pay more attention to $\det M_k$ in the following proof of our main theorem.

\section{Type $C$}

\subsection{Special odd weights of type $C$}
\setcounter{theorem}{1}
Let  $\mathfrak g$ be the simple Lie algebra of type $C_n$ over $\mathbf k$. As mentioned before,
we can identify $\Delta^o$ with the set of sequences $(i_1, i_2,\cdots, i_k)$ of integers, satisfying the following conditions (1)-(4):
\begin{enumerate}
\item $1 \leq i_1<i_2< \cdots <i_k \leq n-1$,
\item $s_l \leq s_{l+2}$ for $l \in [1,k-2]$,
\item $s_l \equiv 0$ for $l \in [1,k]$ with $l \equiv k-1$,
\item $s_{k-1} \leq 2(n-i_k)$,
\end{enumerate}
where $s_1 =i_1$ and $s_l = i_l-i_{l-1}$ for $l \in [2,k]$,
through identifying $d \in \Delta^o$ with $(i_1,i_2,\cdots,i_k)$ where $\{i_l | l \in [1,k]\} = \{ i \in [1,n] | d(\alpha_i) =1\}$.

\medskip

Let $d = (i_1,i_2,\cdots,i_k) \in \Delta^o$. As before, we set $s_1 =i_1$ and $s_l = i_l-i_{l-1}$ for $l \in [2,k]$. By the condition $(\spadesuit_C)$ in Section \ref{WDD}, it is not difficult to see that  $d$ is special if it satisfies the following conditions.

\noindent (a) $s_1$ is even;

\noindent (b) For any two integers $u\leq v$ in $[1,k]$, if $u\equiv v \equiv k-1$ and $$s_{u-2}<s_u=s_{u+2}=\dots=s_{v-2}=s_v<s_{v+2},$$ then $s_{v+1}-s_{u-1}$ is even.

\begin{lemma}
Let $d = (r) \in \Delta^o$.
\begin{enumerate}
\item If $d \in \Delta_{\text{spec}}$, then there exists a non-degenerate $\lambda \in \operatorname{Hom}_{\mathbb Z}( \mathfrak g_{\mathbb Z,d}(2),\mathbb Z)$.
\item If $d \not\in \Delta_{\text{spec}}$ and $\operatorname{char} \mathbf k =2$, then $\det \mathscr G_{d,\lambda} =0$ for any $\lambda \in  \operatorname{Hom}_{\mathbf k}(\mathfrak g_{d}(2), \mathbf k)$.
\end{enumerate}
\end{lemma}

\begin{proof} When $d = (r) \in \Delta^o$,
$$\Phi_{d,1}=\{\varepsilon_i-\varepsilon_j, \varepsilon_i+\varepsilon_j\mid 1\leq i \leq r, r< j\leq n\}.$$
Now
we put
$$A_{\lambda,k}=(\lambda([e_{\varepsilon_i-\varepsilon_k},e_{\varepsilon_j+\varepsilon_k}])_{(i,j)\in [1,r]\times[1,r]}$$
for $\lambda \in \operatorname{Hom}_{\mathbb Z}( \mathfrak g_{\mathbb Z,d}(2),\mathbb Z)$, or $\lambda \in \operatorname{Hom}_{\mathbf k}( \mathfrak g_{ d}(2),\mathbf k)$, and ${k\in [r+1,n]}$.
Then the Gram matrix is
\begin{eqnarray}
\mathscr G_{d,\lambda} =
\begin{pmatrix}
& A \\
-{^tA} &
\end{pmatrix},
\end{eqnarray}
where $A = \operatorname{Diag}(A_{\lambda,r+1}, \cdots, A_{\lambda,n})$.

(1) If $d$ is special, then $r$ is even. In this case, we take
$\lambda \in \operatorname{Hom}_{\mathbb Z}( \mathfrak g_{\mathbb Z,d}(2),\mathbb Z)$ satisfying $\lambda(e_{\varepsilon_i+\varepsilon_{i+1}})=1$ for $i\in [1,r]$, $i\equiv 1$, and $\lambda(e_\alpha)=0$ for other roots.
Then it is easy to check that $\det \mathscr G_{d,\lambda} \in \{ \pm 1\}$ for this $\lambda$.

(2) If $d$ is not special, then $r$ is odd. Assume that $\text{char}~\mathbf k =2 $, therefore for each $k\in [r+1,n]$, $A_{\lambda, k}$ is an symmetric matrix of odd rank whose diagonal is zero.
 Consequently, $\det A_{\lambda, k} =0$ for any $\lambda \in  \operatorname{Hom}_{\mathbf k}(\mathfrak g_{d}(2), \mathbf k)$. So $\det \mathscr G_{d,\lambda} =0$ for any $\lambda \in  \operatorname{Hom}_{\mathbf k}(\mathfrak g_{d}(2), \mathbf k)$.

\end{proof}

\begin{lemma}\label{C(1)}
Let $d = (i_1,i_2, \cdots, i_k) \in \Delta^o \cap \Delta_{\text{spec}}$. Then there exists a non-degenerate $\lambda \in \operatorname{Hom}_{\mathbb Z}( \mathfrak g_{\mathbb Z,d}(2),\mathbb Z)$.
\end{lemma}

\begin{proof}
We look for a faithful map $\lambda \in \operatorname{Hom}_{\mathbb Z}( \mathfrak g_{\mathbb Z,d}(2),\mathbb Z)$ such that  $\lambda$ is non-degenerate. Without loss of generality, we can assume that $k$ is even. The discussion for $k$ odd is similar.

For $l\in [1,k-2]$, we put
$$X_{l}=\{\varepsilon_s-\varepsilon_t\in \Phi_{d,2}\mid s\leq i_{l}<i_{l+1}<t~\text{and}~ t-s=i_{l+1}-i_{l-1}\}.$$
We also let
\begin{eqnarray}
X_{k-1} & = & \{\varepsilon_j+\varepsilon_{j+1}\mid j\in [i_{k-1}+1,i_{k}-1]~\text{and}~ j\equiv i_{k-1}+1\}, \nonumber \\
X_k & = & \{ \varepsilon_s + (-1)^{s-i_{k-2}} \varepsilon_{i_k+\lfloor\frac{s-i_{k-2}+1}{2}\rfloor}\mid s\in [i_{k-2}+1,i_{k-1}]\}. \nonumber
\end{eqnarray}
Denote $\mathbb{X} = \cup_{l = 1}^k X_{l} \subseteq \Phi_{d,2}$.

Let $\lambda: \mathfrak g_{\mathbb Z,d}(2) \rightarrow \mathbb Z$ be the homomorphism such that
$$\lambda(e_{\alpha}) =
\begin{cases}
1 & \text{ if } \alpha \in \mathbb{X}, \\
0 & \text{ if } \alpha \in \Phi_{d,2} \backslash\mathbb{X}.
\end{cases}$$
Then $\lambda$ is faithful.
Recall the definitions of $M_{l+1}$ for $l \in [1,k-1]$ in Section \ref{ssfai}.
It is easy to check that $\det M_{l+1}\in \{\pm 1\}$ for $l\in [1,k-2]$.
In the following we show that $\det M_k \in \{ \pm 1\}$ which implies that $\det \mathscr G_{d,\lambda} \in \{ \pm 1\}$, by (\ref{e_prod}).

To compute $ \det M_k$, we need to consider the set $Q_{k-1}$, as defined in Section \ref{ssfai}, which has the following form.
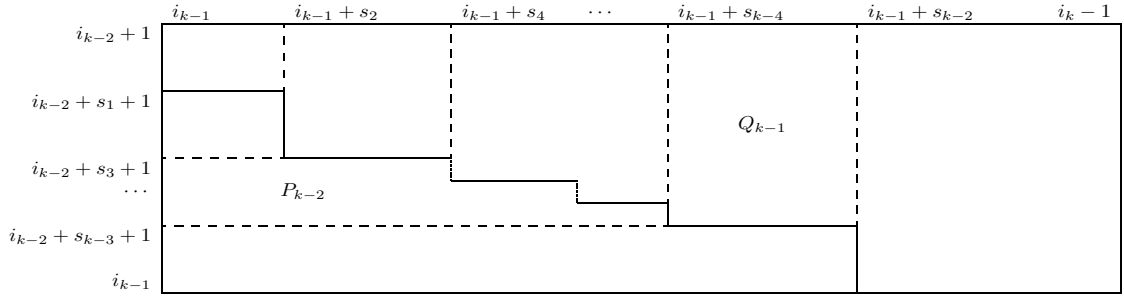
\begin{center}
\resizebox{\textwidth}{20mm}{ 
\begin{tabular}{r cccc cccc ccccc cccc ccccc}
    & $i_{k-1}$ & & & &
    $i_{k-1}+s_2$  &  & & &
    $i_{k-1}+s_4$ &&  $\cdots$ &&&
    $i_{k-1}+s_{k-4}$ & &&&
    $i_{k-1}+s_{k-2}$ &&&& $i_{k}-1$ \\
    \cline{2-23}
$i_{k-2}+1$ & \multicolumn{4}{|c;{4pt/4pt}}{\multirow{3}{*}{   }} &  \multicolumn{4}{c;{4pt/4pt}}{\multirow{3}{*}{  }} &
\multicolumn{5}{c;{4pt/4pt}}{\multirow{8}{*}{  }} &
\multicolumn{4}{c;{4pt/4pt}}{\multirow{9}{*}{ $Q_{k-1}$   }} &
\multicolumn{5}{c|}{\multirow{12}{*}{  }}
\\
& \multicolumn{4}{|c;{4pt/4pt}}{\multirow{3}{*}{ }} &
\multicolumn{4}{c;{4pt/4pt}}{\multirow{3}{*}{  }} &
\multicolumn{5}{c;{4pt/4pt}}{\multirow{8}{*}{  }} &
\multicolumn{4}{c;{4pt/4pt}}{\multirow{9}{*}{  }} &
\multicolumn{5}{c|}{\multirow{12}{*}{  }}
\\
& \multicolumn{4}{|c;{4pt/4pt}}{\multirow{3}{*}{ }} &
\multicolumn{4}{c;{4pt/4pt}}{\multirow{3}{*}{  }} &
\multicolumn{5}{c;{4pt/4pt}}{\multirow{8}{*}{  }} &
\multicolumn{4}{c;{4pt/4pt}}{\multirow{9}{*}{  }} &
\multicolumn{5}{c|}{\multirow{12}{*}{  }}
\\

\cline{2-5}
$i_{k-2}+s_{1}+1$ & \multicolumn{4}{|c|}{\multirow{3}{*}{ }} &
\multicolumn{4}{c;{4pt/4pt}}{\multirow{3}{*}{ }} &
\multicolumn{5}{c;{4pt/4pt}}{\multirow{8}{*}{ }} &
\multicolumn{4}{c;{4pt/4pt}}{\multirow{9}{*}{ }} &
\multicolumn{5}{c|}{\multirow{12}{*}{ }}
\\
& \multicolumn{4}{|c|}{\multirow{3}{*}{ }} &
\multicolumn{4}{c;{4pt/4pt}}{\multirow{3}{*}{ }} &
\multicolumn{5}{c;{4pt/4pt}}{\multirow{8}{*}{ }} &
\multicolumn{4}{c;{4pt/4pt}}{\multirow{9}{*}{ }} &
\multicolumn{5}{c|}{\multirow{12}{*}{ }}
\\
& \multicolumn{4}{|c|}{\multirow{3}{*}{ }} &
\multicolumn{4}{c;{4pt/4pt}}{\multirow{3}{*}{ }} &
\multicolumn{5}{c;{4pt/4pt}}{\multirow{8}{*}{ }} &
\multicolumn{4}{c;{4pt/4pt}}{\multirow{9}{*}{ }} &
\multicolumn{5}{c|}{\multirow{12}{*}{ }}
\\

\cdashline{2-5} \cline{6-9}
$i_{k-2}+s_3+1$ & \multicolumn{8}{|c;{1pt/1pt}}{\multirow{3}{*}{$P_{k-2}$ }} &
\multicolumn{5}{c;{4pt/4pt}}{\multirow{8}{*}{ }} &
\multicolumn{4}{c;{4pt/4pt}}{\multirow{9}{*}{ }} &
\multicolumn{5}{c|}{\multirow{12}{*}{ }}
\\
\cline{10-11}
$\cdots$  & \multicolumn{10}{|c;{1pt/1pt}}{\multirow{3}{*}{ }} &
\multicolumn{3}{c;{4pt/4pt}}{\multirow{8}{*}{ }} &
\multicolumn{4}{c;{4pt/4pt}}{\multirow{9}{*}{ }} &
\multicolumn{5}{c|}{\multirow{12}{*}{ }}
\\
\cline{12-14}
 & \multicolumn{13}{|c|}{\multirow{3}{*}{ }} &
\multicolumn{4}{c;{4pt/4pt}}{\multirow{9}{*}{ }} &
\multicolumn{5}{c|}{\multirow{12}{*}{ }}
\\
\cdashline{2-14} \cline{15-18}
$i_{k-2}+s_{k-3}+1$ & \multicolumn{17}{|c|}{\multirow{3}{*}{ }} &
\multicolumn{5}{c|}{\multirow{12}{*}{ }}
\\
& \multicolumn{17}{|c|}{\multirow{3}{*}{ }} &
\multicolumn{5}{c|}{\multirow{12}{*}{ }}
\\
$i_{k-1}$ & \multicolumn{17}{|c|}{\multirow{3}{*}{ }} &
\multicolumn{5}{c|}{\multirow{12}{*}{ }}
\\
\cline{2-23}
\end{tabular}
} 
\end{center}
For each $j\in [i_{k-2}+1,i_{k-1}]$, there exit two odd integers $u \leq v$ in $[1,k-1]$ satisfying $$s_{u-2}<s_u=s_{u+2}=\dots=s_{v-2}=s_v<s_{v+2},$$ and $j\in [i_{k-2}+s_{u-2}+1, i_{k-2}+s_v]$.
Set $$\Theta_j=[i_{k-1}+s_{u-1}, i_{k}-1].$$
For such $u,v$, $s_{v+1}-s_{u-1}$ is even, since $d\in \Delta^o$ is special and $k$ is even by the assumption.
It is easy to see that $\Theta_j=\Theta_{j+1}$ when $j\equiv i_{k-2}+1$.

We denote $\lambda(\alpha)=\lambda(e_\alpha)$ for simplicity, 
for $e_\alpha \in {\mathfrak g}_{\mathbb{Z},d}(2)$. Let
$$x_j=\lambda(\varepsilon_j+\varepsilon_{j+1})~\text{and}~y_s= \lambda(\varepsilon_s+(-1)^{s-i_{k-2}}\varepsilon_{i_k+\lfloor\frac{s-i_{k-2}+1}{2}\rfloor})$$
for $j\in [i_{k-1}+1,i_{k}-1]$, $j\equiv i_{k-1}+1$ and $s\in [i_{k-2}+1,i_{k-1}]$.

As in (\ref{e_MK}),
$$ M_k = (\lambda([e_{\alpha},e_{\beta}]))_{\alpha, \beta \in Q_{k-1} \cup Y_{d,k}}.$$
So we divide $Q_{k-1} \cup Y_{d,k}$ into subsets and consider the submatrices of $M_k$ corresponding to these subsets.

\medskip

\noindent (a) Suppose that $s_{k}$ is even. Then $|\Theta_s|$ is even for each $s\in [i_{k-2}+1,i_{k-1}]$.
We can divide $Q_{k-1} \cup Y_{d,k}$ into the following subsets:

\hspace{0.1cm} (a1) $\{\varepsilon_s-\varepsilon_j,\  \varepsilon_j+(-1)^{s-i_{k-2}}\varepsilon_{i_k+\lfloor\frac{s-i_{k-2}+1}{2}\rfloor},\  \varepsilon_{j+1}-(-1)^{s-i_{k-2}}\varepsilon_{i_k+\lfloor\frac{s-i_{k-2}+1}{2}\rfloor},\ \varepsilon_{s+1}-\varepsilon_{j+1} \}$, for $j,j+1 \in \Theta_s $;

\hspace{0.1cm} (a2) $\{\varepsilon_s-\varepsilon_{j+1}, \ \varepsilon_{j+1}+(-1)^{s-i_{k-2}}\varepsilon_{i_k+\lfloor\frac{s-i_{k-2}+1}{2}\rfloor}, \ \varepsilon_{j}-(-1)^{s-i_{k-2}}\varepsilon_{i_k+\lfloor\frac{s-i_{k-2}+1}{2}\rfloor},\ \varepsilon_{s+1}-\varepsilon_{j}\}$
for $j,j+1 \in \Theta_s $;

\hspace{0.1cm} (a3) the remaining roots can be partitioned into subsets of the form
$\{\varepsilon_j-\varepsilon_l,\ \varepsilon_{j+1}+\varepsilon_l,\  \varepsilon_j+\varepsilon_l, \ \varepsilon_{j+1}-\varepsilon_l \}.$

Let $X$ be a subset of $Q_{k-1} \cup Y_{d,k}$ of type (a1), (a2) or (a3). Then
$$(\lambda([e_{\alpha},e_{\beta}]))_{\alpha, \beta \in X} =
\begin{cases}
\begin{pmatrix}
&0   & y_s   & 0  & 0  \\
&-y_s  & 0  & x_j  & 0   \\
&0   & -x_j & 0   & y_{s+1}     \\
&0   & 0    & -y_{s+1}  & 0
\end{pmatrix} & \text{ if } X \text{ is of type (a1) or (a2)}, \\
\begin{pmatrix}
&0   &y_j  &0   &0 \\
&-y_j  &0  &0   &0 \\
&0   &0   &0   &y_j \\
 &0   &0 &-y_j  &0
\end{pmatrix} & \text{ if } X \text{ is of type (a3)}.
\end{cases}$$

\medskip

\noindent (b) Suppose $s_{k}$ is odd. Then $s_{k-1}=2(n-i_k)$ and
$|\Theta_s|$ is  odd for each $s\in [i_{k-2}+1,i_{k-1}]$.
Besides the subsets given as before, there is still one more type of roots, namely,
$$\{\varepsilon_s-\varepsilon_{i_k},\ \varepsilon_{i_k}+(-1)^{s-i_{k-2}}\varepsilon_{i_k+\lfloor\frac{s-i_{k-2}+1}{2}\rfloor},\  \varepsilon_{s+1}-\varepsilon_{i_k},\ \varepsilon_{i_k}-(-1)^{s-i_{k-2}}\varepsilon_{i_k+\lfloor\frac{s-i_{k-2}+1}{2}\rfloor} \}$$
for $s\in [i_{k-2}+1,i_{k-1}]$ and $s\equiv i_{k-2}+1$.

It is easy to see that $M_k$ can be written as a blocked diagonal matrix with diagonal blocks of the form $(\lambda([e_{\alpha},e_{\beta}]))_{\alpha, \beta \in X}$ for each subset $X$ of type (a1), (a2),(a3) or (b), and each block is of determinant $\pm 1$.
%
Thus we have proved Theorem \ref{mainThm1} (1) for type $C$.

\end{proof}

\begin{lemma} \label{C(2)}
Let $d = (i_1,i_2, \cdots, i_k) \in \Delta^o \backslash \Delta_{\text{spec}}$.
If $\operatorname{char} \mathbf k =2$, then $\det \mathscr G_{d,\lambda} =0$ for any $\lambda \in \operatorname{Hom}_{\mathbf k}(\mathfrak g_{d}(2), \mathbf k)$.
\end{lemma}

\begin{proof}
Suppose that there exists $\lambda \in \operatorname{Hom}_{\mathbf k}(\mathfrak g_{d}(2), \mathbf k)$  such that $\det \mathscr G_{d,\lambda} \ne 0$. We can assume that $\lambda$ is faithful by
(\ref{e_iso_eq}) and Lemma \ref{l_iso_faith}. 
So by (\ref{e_prod}), $\det \mathscr G_{d,\lambda} \ne 0$ implies that $\det M_k \ne 0$, where $M_k$ is defined as in (\ref{e_MK}). In the following we assume that $k$ is even without loss of generality.

For each $l\in [i_{k-1}+1,i_k]$, there exist two even integers $u \leq v$ in $[0,k-2]$,  which satisfy $$s_{u-2}<s_u=s_{u+2}=\dots=s_{v-2}=s_v<s_{v+2},$$
and $i_{k-1}+s_u< l\leq i_{k-1}+s_v$. We denote $\Xi_l=[i_{k-2}+1, i_{k-2}+s_{v+1}]$.
Since $d$ is not special,
there exist $\zeta<\eta$ in $[1,k]$ which satisfy $$s_{\zeta-2}<s_\zeta=s_{\zeta+2}=\dots=s_\eta<s_{\eta+2},$$ such that $s_\eta-s_\zeta$ is odd and
$\Xi_{s_\zeta+1}=\Xi_{s_\zeta+2}=\dots = \Xi_{s_\eta}$.

For each fixed $s\in [i_{k-1}+1,i_k]$, we set
$$B_s= \lambda ([e_{\varepsilon_r-\varepsilon_s}, e_{\varepsilon_s-\varepsilon_l} ])_{r\in \Xi_s, l\in [i_k+1,n]}\  \ \text{and} \ \  C_s= \lambda ([e_{\varepsilon_r-\varepsilon_s}, e_{\varepsilon_s+\varepsilon_l} ])_{r\in \Xi_s, l\in [i_k+1,n]}.$$
Then we put $A_s=(B_s, C_s)$.

For any $s,t\in [i_{k-1}+1,i_k]$, let $E_{s,t}$ be the submatrix of $M_k$ corresponding to the set of roots
$$\{\varepsilon_s-\varepsilon_l,\  \varepsilon_s+\varepsilon_l, \ \varepsilon_t-\varepsilon_l, \ \varepsilon_t+\varepsilon_l\mid l\in [i_k+1,n] \}.$$

 With a good ordering in $Q_k\cup Y_{d,k}$ the matrix $M_k$ has the same form of $S$  in Proposition 
 \ref{p_det_D}(1). Since $d$ is not special, all conditions in Proposition 
 \ref{p_det_D}(1) are satisfied. Thus we have $\det M_k =0$. We get a contradiction which completes the proof of the lemma.
\end{proof}

Combining Lemma \ref{C(1)} and Lemma \ref{C(2)}, we have proved Theorem \ref{mainThm2} in the type $C$ case.

\section{Type $B$}

\setcounter{theorem}{1}
\subsection{Special odd weights of type $B$}
Let $\mathfrak g$ be the simple Lie algebra of type $B_n$ over $\mathbf k$.
As before, we can identify $\Delta^o$ with the set of sequences $(i_1, i_2,\cdots, i_k)$ of integers, satisfying the following (1)-(4):
\begin{enumerate}
\item $1 \leq i_1<i_2< \cdots <i_k \leq n$,
\item $s_l \leq s_{l+2}$ for $l \in [1,k-2]$,
\item $s_l \equiv 0$ for $l \in [1,k]$ with $l \equiv k$,
\item $s_{k-1} \leq 2(n-i_k)+1$,
\end{enumerate}
where $s_1 =i_1$ and $s_l = i_l-i_{l-1}$ for $l \in [2,k]$,
through identifying $d \in \Delta^o$ with $(i_1,i_2,\cdots,i_k)$ where $\{i_l| l \in [1,k]\} = \{ i \in [1,n]| d(\alpha_i) =1\}$.

\medskip

Let $d = (i_1,i_2,\cdots,i_k) \in \Delta^o$. As before, we set $s_1 =i_1$ and $s_l = i_l-i_{l-1}$ for $l \in [2,k]$. By the condition $(\spadesuit_B)$ in Section \ref{WDD}, it is not difficult to see that  $d\in \Delta^{o}$ is special if it satisfies  following conditions.

\noindent (a) $k$ is even and $s_1$ is odd;

\noindent (b) For any two even integers $u \leq v$ in $[1,k-2]$, if $$s_{u-2}<s_u=s_{u+2}=\dots=s_{v-2}=s_v<s_{v+2},$$ then $s_{v+1}-s_{u-1}$ is even.

\begin{lemma}
Let $d = (r) \in \Delta^o$, which is not special. If $\operatorname{char} \mathbf k =2$, then $\det \mathscr G_{d,\lambda} =0$ for any $\lambda \in \operatorname{Hom}(\mathfrak g_{d}(2), \mathbf k)$.
\end{lemma}
\begin{proof} We have $[e_{\varepsilon_i-\varepsilon_k},e_{\varepsilon_j+\varepsilon_k}]=\pm e_{\varepsilon_i+\varepsilon_j}$ and $[e_{\varepsilon_i}, e_{\varepsilon_j}]=\pm 2e_{\varepsilon_i+\varepsilon_j}$ for $i\ne j \in [1,n]$. So for $\lambda \in \operatorname{Hom}(\mathfrak g_{d}(2), \mathbf k)$,  we denote
$$A_{\lambda,k}=(\lambda([e_{\varepsilon_i-\varepsilon_k},e_{\varepsilon_j+\varepsilon_k}]))_{(i,j)\in [1,r]\times[1,r]}$$
for ${k\in [r+1,n]}$. Set $b_{ij}=\lambda([e_{\varepsilon_i}, e_{\varepsilon_j}])$ for $i, j\in [1,r]$. Thus the matrix $\mathscr G_{d,\lambda}$ is
$$\begin{pmatrix}
&0              & A        & 0      \\
&-{^tA}          & 0       & 0      \\
&0               & 0       & B
\end{pmatrix},$$
where $A = \operatorname{Diag}(A_{\lambda,r+1}, \cdots, A_{\lambda,n})$ and $B=(b_{ij})$. Thus
$\det \mathscr G_{d,\lambda} = 0$ as $\det B = 0$.

\end{proof}

\begin{lemma}\label{B(1)}
Let $d = (i_1,i_2, \cdots, i_k) \in \Delta^o \cap \Delta_{\text{spec}}$. Then there exists a non-degenerate $\lambda \in \operatorname{Hom}_{\mathbb Z}( \mathfrak g_{\mathbb Z,d}(2),\mathbb Z)$.
\end{lemma}
\begin{proof} Since $d$ is special we see that $k$ is even.
We look for the  $\lambda$ such that  $\lambda$ is faithful.
For $k\geq 2$, following the strategy in Section \ref{s_faithful_map},
we set
$$M_{l+1} = (\lambda([e_{\alpha},e_{\beta}]))_{\alpha \in Q_l,\beta \in P_l},~~\text{for $l\in [1,k-2]$}$$
and $M_k = (\lambda([e_{\alpha},e_{\beta}]))_{\alpha, \beta \in Q_{k-1} \cup Y_{d,k}}$. Then by (\ref{e_prod}),
$$\det \mathscr G_{d,\lambda} = \pm (\prod_{l=2}^{k-1} \det M_l)^2 \det M_k.$$

For $l\in [1,k-2]$, put
$$X_{l}=\{\varepsilon_s-\varepsilon_t\in \Phi_{d,2}\mid s\leq i_{l}<i_{l+1}<t~\text{and}~ t-s=i_{l+1}-i_{l-1}\} ,$$
and
$$X_{k-1}=\{\varepsilon_j+\varepsilon_{j+1} \in\Phi_{d,2}\mid j\in [i_{k-1}+1,i_{k}-1]~\text{and}~ j\equiv i_{k-1}+1\} .$$
To compute $ \det M_k$, we need to consider the set $Q_{k-1}$ which has the following form

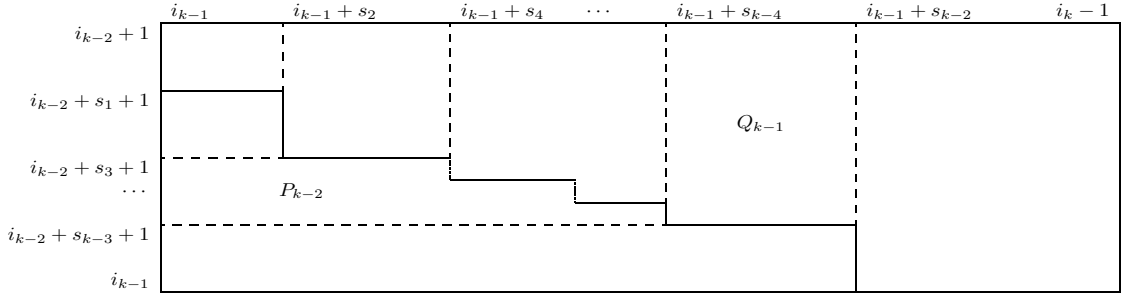
\begin{center}
\resizebox{\textwidth}{20mm}{
\begin{tabular}{r cccc cccc ccccc cccc ccccc}
    & $i_{k-1}$ & & & &
    $i_{k-1}+s_2$  &  & & &
    $i_{k-1}+s_4$ &&  $\cdots$ &&&
    $i_{k-1}+s_{k-4}$ & &&&
    $i_{k-1}+s_{k-2}$ &&&& $i_{k}-1$ \\
    \cline{2-23}
$i_{k-2}+1$ & \multicolumn{4}{|c;{4pt/4pt}}{\multirow{3}{*}{   }} &  \multicolumn{4}{c;{4pt/4pt}}{\multirow{3}{*}{  }} &
\multicolumn{5}{c;{4pt/4pt}}{\multirow{8}{*}{  }} &
\multicolumn{4}{c;{4pt/4pt}}{\multirow{9}{*}{ $Q_{k-1}$   }} &
\multicolumn{5}{c|}{\multirow{12}{*}{  }}
\\
& \multicolumn{4}{|c;{4pt/4pt}}{\multirow{3}{*}{ }} &
\multicolumn{4}{c;{4pt/4pt}}{\multirow{3}{*}{  }} &
\multicolumn{5}{c;{4pt/4pt}}{\multirow{8}{*}{  }} &
\multicolumn{4}{c;{4pt/4pt}}{\multirow{9}{*}{  }} &
\multicolumn{5}{c|}{\multirow{12}{*}{  }}
\\
& \multicolumn{4}{|c;{4pt/4pt}}{\multirow{3}{*}{ }} &
\multicolumn{4}{c;{4pt/4pt}}{\multirow{3}{*}{  }} &
\multicolumn{5}{c;{4pt/4pt}}{\multirow{8}{*}{  }} &
\multicolumn{4}{c;{4pt/4pt}}{\multirow{9}{*}{  }} &
\multicolumn{5}{c|}{\multirow{12}{*}{  }}
\\

\cline{2-5}
$i_{k-2}+s_{1}+1$ & \multicolumn{4}{|c|}{\multirow{3}{*}{ }} &
\multicolumn{4}{c;{4pt/4pt}}{\multirow{3}{*}{ }} &
\multicolumn{5}{c;{4pt/4pt}}{\multirow{8}{*}{ }} &
\multicolumn{4}{c;{4pt/4pt}}{\multirow{9}{*}{ }} &
\multicolumn{5}{c|}{\multirow{12}{*}{ }}
\\
& \multicolumn{4}{|c|}{\multirow{3}{*}{ }} &
\multicolumn{4}{c;{4pt/4pt}}{\multirow{3}{*}{ }} &
\multicolumn{5}{c;{4pt/4pt}}{\multirow{8}{*}{ }} &
\multicolumn{4}{c;{4pt/4pt}}{\multirow{9}{*}{ }} &
\multicolumn{5}{c|}{\multirow{12}{*}{ }}
\\
& \multicolumn{4}{|c|}{\multirow{3}{*}{ }} &
\multicolumn{4}{c;{4pt/4pt}}{\multirow{3}{*}{ }} &
\multicolumn{5}{c;{4pt/4pt}}{\multirow{8}{*}{ }} &
\multicolumn{4}{c;{4pt/4pt}}{\multirow{9}{*}{ }} &
\multicolumn{5}{c|}{\multirow{12}{*}{ }}
\\

\cdashline{2-5} \cline{6-9}
$i_{k-2}+s_3+1$ & \multicolumn{8}{|c;{1pt/1pt}}{\multirow{3}{*}{$P_{k-2}$ }} &
\multicolumn{5}{c;{4pt/4pt}}{\multirow{8}{*}{ }} &
\multicolumn{4}{c;{4pt/4pt}}{\multirow{9}{*}{ }} &
\multicolumn{5}{c|}{\multirow{12}{*}{ }}
\\
\cline{10-11}
$\cdots$  & \multicolumn{10}{|c;{1pt/1pt}}{\multirow{3}{*}{ }} &
\multicolumn{3}{c;{4pt/4pt}}{\multirow{8}{*}{ }} &
\multicolumn{4}{c;{4pt/4pt}}{\multirow{9}{*}{ }} &
\multicolumn{5}{c|}{\multirow{12}{*}{ }}
\\
\cline{12-14}
 & \multicolumn{13}{|c|}{\multirow{3}{*}{ }} &
\multicolumn{4}{c;{4pt/4pt}}{\multirow{9}{*}{ }} &
\multicolumn{5}{c|}{\multirow{12}{*}{ }}
\\
\cdashline{2-14} \cline{15-18}
$i_{k-2}+s_{k-3}+1$ & \multicolumn{17}{|c|}{\multirow{3}{*}{ }} &
\multicolumn{5}{c|}{\multirow{12}{*}{ }}
\\
& \multicolumn{17}{|c|}{\multirow{3}{*}{ }} &
\multicolumn{5}{c|}{\multirow{12}{*}{ }}
\\
$i_{k-1}$ & \multicolumn{17}{|c|}{\multirow{3}{*}{ }} &
\multicolumn{5}{c|}{\multirow{12}{*}{ }}
\\
\cline{2-23}
\end{tabular}
}
\end{center}

\smallskip

Since $d\in \Delta^o$ is special, if there exists two even integers $u \leq v$ in $[0,k-2]$ satisfying
\begin{eqnarray} \label{e_cond_s}
s_{u-2}<s_u=s_{u+2}=\dots=s_{v-2}=s_v<s_{v+2},
\end{eqnarray}
then $s_{v+1}-s_{u-1}$ is even.
For $u \leq v$ in $[0,k-2]$ satisfying (\ref{e_cond_s}), we denote
$$\Omega_{u,v}=[i_{k-2}+s_{u-1}+1,\ i_{k-2}+s_{v+1} ],$$
and set
$$\Omega=\bigcup_{u \leq v \text{ in } [0,k-2] \text{ with (\ref{e_cond_s})}} \Omega_{u,v} \cup[i_{k-2}+2,i_{k-2}+s_1]. $$
The cardinality $\omega=|\Omega|$ is even. We denote
$$\Omega=\{j_1,j_2,\dots, j_\omega\}.$$
with $j_{r}<j_{s}$ for $r<s$ and set

$$X_{k}= \{\varepsilon_{j_s}+(-1)^s \varepsilon_{i_k+\lfloor\frac{s+1}{2}\rfloor}\mid s\in [1,\omega]\}.$$
With these settings, we denote $\mathbb{X} =\cup_{l = 1}^k X_{l} \cup \{\varepsilon_{i_{k-2}+1}\}$  which is a subset of $\Phi_{d,2}$.

Let $\lambda: \mathfrak g_{\mathbb Z,d}(2) \rightarrow \mathbb Z$ be the homomorphism such that
$$\lambda(e_{\alpha}) =
\begin{cases}
1 & \text{ if } \alpha \in \mathbb{X} , \\
0 & \text{ if } \alpha \in \Phi_{d,2} \backslash \mathbb{X} .
\end{cases}$$
In the following we show that $\det \mathscr G_{d,\lambda} \in \{ \pm 1\}$.

Obviously, $\lambda$ is faithful in the sense of Section \ref{ssfai}, then
by (\ref{e_prod}),
$$\det \mathscr G_{d,\lambda} = \pm (\prod_{l=2}^{k-1} \det M_l)^2 \det M_k.$$
It is easy to check that
$\det M_l \in \{ \pm 1\}$ for $l \in [2,k-1]$.

To compute $\det M_k$, we put
\begin{eqnarray}
x& = & \lambda(\varepsilon_{i_{k-2}+1}), \nonumber \\
y_r& = &\lambda(\varepsilon_r+\varepsilon_{r+1}), \nonumber \\
z_s & = & \lambda(\varepsilon_{j_s}+(-1)^s \varepsilon_{i_k+\lfloor\frac{s+1}{2}\rfloor}), \nonumber
\end{eqnarray}
for $r\in [i_{k-1}+1,i_{k}-1]$ with $r\equiv i_{k-1}+1$, and $s\in [1,\omega]$.

As
$$ M_k = (\lambda([e_{\alpha},e_{\beta}]))_{\alpha, \beta \in Q_{k-1} \cup Y_{d,k}},$$
we divide $Q_{k-1} \cup Y_{d,k}$ into disjoint subsets and consider the submatrices of $M_k$ corresponding to these subsets.

\medskip

\noindent (a) The first type of subsets are of the form
$$\{\varepsilon_{i_{k-2}+1}-\varepsilon_r,\ \varepsilon_r, \ \varepsilon_{r+1},\  \varepsilon_{i_{k-2}+1}-\varepsilon_{r+1} \}$$
for $r\in [i_{k-1}+1,i_{k}-1]$ with $r\equiv i_{k-1}+1$.

\medskip

\noindent (b) The second type of subsets are of the form
$$\{\varepsilon_{j_s}-\varepsilon_l,\  \varepsilon_{l}+(-1)^s \varepsilon_{i_k+\lfloor\frac{s+1}{2}\rfloor},
\ \varepsilon_{l+1}+(-1)^{s+1} \varepsilon_{i_k+\lfloor\frac{s}{2}\rfloor+1},
\ \varepsilon_{j_{s+1}}-\varepsilon_{l+1}\},$$
or
$$\{\varepsilon_{j_s}-\varepsilon_{l+1},\  \varepsilon_{l+1}+(-1)^s \varepsilon_{i_k+\lfloor\frac{s+1}{2}\rfloor},
\ \varepsilon_{l}+(-1)^{s+1} \varepsilon_{i_k+\lfloor\frac{s}{2}\rfloor+1},
\ \varepsilon_{j_{s+1}}-\varepsilon_{l}\},$$
where $s\in [1,\omega]$ and $l\in [i_{k-1}+1,i_{k}-1]$ satisfy
$s\equiv 1$, $l\equiv i_{k-1}+1$ and $\varepsilon_{j_\sigma}-\varepsilon_\tau \in Q_{k-1}$ for $\sigma\in [1,\omega]$ and $\tau\in [i_{k-1}+1,i_{k}-1]$.

\medskip

\noindent (c) The set of the remaining roots in $Q_{k-1} \cup Y_{d,k}$ can be divided into subsets of the form
$$\{\varepsilon_j-\varepsilon_l,\ \varepsilon_{j+1}+\varepsilon_l,\  \varepsilon_j+\varepsilon_l, \ \varepsilon_{j+1}-\varepsilon_l \}.$$

\medskip

Let $X$ be a subset of $Q_{k-1} \cup Y_{d,k}$ of type (a), (b) or (c). Then
\begin{eqnarray}
(\lambda([e_{\alpha},e_{\beta}]))_{\alpha, \beta \in X}  =
\begin{cases}
\begin{pmatrix}
&0   & x   & 0  & 0  \\
&-x  & 0   & 2y_r   & 0   \\
&0   & -2y_r    & 0  & x   \\
&0   & 0    & -x  & 0
\end{pmatrix} & \text{ if } X \text{ is of type (a)}, \\
\begin{pmatrix}
&0   & z_s   & 0  & 0  \\
&-z_s  & 0   & y_l   & 0   \\
&0   & -y_l    & 0  & z_{s+1}  \\
&0   & 0    & -z_{s+1}  & 0
\end{pmatrix} & \text{ if } X \text{ is of type (b)}, \\
\begin{pmatrix}
&0   &y_j  &0   &0 \\
&-y_j  &0  &0   &0 \\
&0   &0   &0   &y_j \\
 &0   &0 &-y_j  &0
\end{pmatrix} & \text{ if } X \text{ is of type (c)}.
\end{cases} \nonumber
\end{eqnarray}

Then $M_k$ can be written as a blocked diagonal matrix with diagonal blocks of the form $(\lambda([e_{\alpha},e_{\beta}]))_{\alpha, \beta \in X}$ for each subset $X$ of type (a), (b) or (c), and each block is of determinant $\pm 1$.
We see that $\det M_k \in \{ \pm 1\}$ which implies $\det \mathscr G_{d,\lambda} \in \{ \pm 1\}$. Thus we have proved Theorem \ref{mainThm1} (1) for type $B$.

\end{proof}

\begin{lemma} \label{B(2)}
Let $d = (i_1,i_2, \cdots, i_k) \in \Delta^o \backslash \Delta_{\text{spec}}$.
If $\operatorname{char} \mathbf k =2$, then we have $\det \mathscr G_{d,\lambda} =0$ for any $\lambda \in \operatorname{Hom}(\mathfrak g_{d}(2), \mathbf k)$
 .
\end{lemma}

\begin{proof}
Suppose that there exists $\lambda \in \operatorname{Hom}(\mathfrak g_{d}(2), \mathbf k)$  such that $\det \mathscr G_{d,\lambda} \ne 0$. Through the arguments as in Section \ref{s_faithful_map},  we can assume that $\lambda$  is faithful. So by (\ref{e_prod}), $\det \mathscr G_{d,\lambda} \ne 0$ implies that $\det M_k \ne 0$, where $M_k$ is defined as in (\ref{e_MK}).

We can assume that $k$ is even. When $k$ is odd, the shape of $Q_{k-1}$ is different
, but the argument is similar.
For each $l\in [i_{k-1}+1,i_k]$, there exist two even integers $u,v\in [2,k-2]$ which satisfy
$u<v$ and
$$s_{u-2}<s_u=s_{u+2}=\dots=s_{v-2}=s_v<s_{v+2},$$
with $i_{k-1}+s_u< l\leq i_{k-1}+s_v$. We denote $\Xi_l=[i_{k-2}+1, i_{k-2}+s_{v+1}]$.
Since $d$ is not special, there exists an integer $l\in [i_{k-1}+1,i_k]$ such that $|\Xi_l|$ is even.

For each fixed $s\in [i_{k-1}+1,i_k]$,
we set $$A_{s}=\lambda([e_{\varepsilon_r-\varepsilon_s},  e_\varpi ])$$
where $r\in \Xi_s $ and $\varpi$ runs over the set
$\{  \varepsilon_s, \ \varepsilon_s-\varepsilon_l,  \ \varepsilon_s+\varepsilon_l \mid   l\in [i_k+1,n]\}.$

For any $s,t\in [i_{k-1}+1,i_k]$, denote $E_{s,t}$ the block associate the set of roots
$$\{\varepsilon_s-\varepsilon_l,\  \varepsilon_s+\varepsilon_l, \ \varepsilon_t-\varepsilon_l, \ \varepsilon_t+\varepsilon_l\mid l\in [i_k+1,n] \}.$$

With a good ordering in $Q_k\cup Y_{d,k}$ the matrix $M_k$  has the same form of $S$ in the Proposition 
\ref{p_det_D}(3). Since $d$ is not special, there exists $s\in [i_{k-1}+1,i_k]$ such that $|\Xi_s|$ is even.  Then all conditions in Proposition 
 \ref{p_det_D}(3) are satisfied and  we have $\det M_k =0$.  We get a contradiction which proves the lemma.

\end{proof}

Combining Lemma \ref{B(1)} and Lemma \ref{B(2)}, we have proved Theorem \ref{mainThm2} in the type $B$ case.

\section{The first Case  of type $D$}

\setcounter{theorem}{1}
\subsection{Odd weights of the first case}
\setcounter{theorem}{1}
Let  $\mathfrak g$ be the simple Lie algebra of type $D_n$ over $\mathbf k$.
As before, we can write $\Delta^o = \Delta^o_1 \sqcup \Delta^o_2$, where $\Delta^o_1$ (resp. $\Delta^o_2$) consists of weighted Dynkin diagrams of the form (\ref{d_D1_1}) (resp. (\ref{d_D2_1})).
In this section and the next section, we consider $\Delta^o_1$ and $\Delta^o_2$ respectively.

The weighted Dynkin diagram  associated to the partition $(2^{2m}1^2)$ is special, in this case it is not difficult to see that there exists a non-degenerate $\lambda \in \operatorname{Hom}_{\mathbb Z}( \mathfrak g_{\mathbb Z,d}(2),\mathbb Z)$. Thus we exclude this situation in the following of this section.

The set $\Delta^o_1$ can be identified with the set of sequences $(i_1, i_2,\cdots, i_k)$ of integers, satisfying the following conditions (1)-(4):
\begin{enumerate}
\item $1 \leq i_1<i_2< \cdots <i_k \leq n-3$,
\item $s_l \leq s_{l+2}$ for $l \in [1,k-1]$,
\item $s_l \equiv 0$ for $l \in [1,k+1]$ with $l \equiv k+1$,
\item $s_l\leq 2$ for $l \in [1,k]$ with $l \equiv k$.
\end{enumerate}
where $s_1 =i_1$ and $s_l = i_l-i_{l-1}$ for $l \in [2,k]$ and $s_{k+1}=n-1-i_k$,
through identifying 
$d \in \Delta^o_1$ with $(i_1,i_2,\cdots,i_k)$ where $\{i_l| l \in [1,k]\} = \{ i \in [1,n]| d(\alpha_i) =1\}$.

Let $d = (i_1,i_2,\cdots,i_k) \in \Delta^o_1$. As before, we set $s_1 =i_1$, $s_l = i_l-i_{l-1}$ for $l \in [2,k]$ and $s_{k+1}=n-1-i_k$. By the condition $(\spadesuit_D)$ in Section \ref{WDD}, it is not difficult to see that  $d$ is special if and only if it satisfies  following conditions:

\noindent (a) $s_1$ is even;

\noindent (b) For any two integers $u\leq v\in [1,k]$ and $u\equiv v \equiv k-1$, if $$s_{u-2}<s_u=s_{u+2}=\dots=s_{v-2}=s_v<s_{v+2},$$ then $s_{v+1}-s_{u-1}$ is even.

\begin{lemma}\label{Dcase1}
Let $d = (i_1,i_2, \cdots, i_k) \in \Delta^o_1$.
\begin{enumerate}
\item If $d \in \Delta_{\text{spec}}$, then there exists a non-degenerate $\lambda \in \operatorname{Hom}_{\mathbb Z}( \mathfrak g_{\mathbb Z,d}(2),\mathbb Z)$.
\item If $\text{char}~\mathbf k =2$ and $d \not\in \Delta_{\text{spec}}$, then $\det \mathscr G_{d,\lambda} =0$ for any $\lambda \in \operatorname{Hom}(\mathfrak g_{d}(2), \mathbf k)$.
\end{enumerate}
\end{lemma}

\begin{proof}
 By the arguments as in Section \ref{s_faithful_map}, we can restrict our attention to faithful $\lambda \in \operatorname{Hom}_{\mathbb Z}( \mathfrak g_{\mathbb Z,d}(2),\mathbb Z)$ or
  $\operatorname{Hom}(\mathfrak g_{d}(2), \mathbf k)$.
 Without loss generality we can assume that $k$ is even. The discussion for $k$ odd is the same.

Let $Y_{k+1}=\{\varepsilon_j-\varepsilon_n, \varepsilon_j+\varepsilon_n \mid j\in [i_k+1, n-1]\}$.
We can also construct the sets $Q_{l+1}, P_l$ for $l\in [1,k-1]$ as before. We set
$$M_{l} = (\lambda([e_{\alpha},e_{\beta}]))_{\alpha \in Q_l,\beta \in P_l},~~\text{for $l\in [1,k-1]$}$$
and $M_k = (\lambda([e_{\alpha},e_{\beta}]))_{\alpha, \beta \in Q_{k} \cup Y_{k+1}}$. Then
$$\det \mathscr G_{d,\lambda} = \pm (\prod_{l=1}^{k-1} \det M_l)^2 \det M_k.$$
So it is enough to prove that:
\begin{enumerate}
\item If $d \in \Delta_{\text{spec}}$, then there exists  $\lambda \in \operatorname{Hom}_{\mathbb Z}( \mathfrak g_{\mathbb Z,d}(2),\mathbb Z)$  such that $ \det M_k=\pm 1$.
 \item If $d \not\in \Delta_{\text{spec}}$ and $\operatorname{char} \mathbf k =2$, then $\det M_k = 0$ for any $\lambda \in \operatorname{Hom}(\mathfrak g_{d}(2), \mathbf k)$.
\end{enumerate}

\noindent (a) If $s_k=1$ and $d \in \Delta_{\text{spec}}$, then $s_1=s_3=\dots=s_{k+1}$. Thus in this case $Q_k$ is empty. So it is very easy to construct the  $\lambda$ such that $ \det M_k=\pm 1$.
When  $d \not\in \Delta_{\text{spec}}$, the set $Q_k$ is
$\{\varepsilon_{i_k}-\varepsilon_j\mid j\in [i_k+s_u+1,n-1]\}$ for some even integer $u$.
For $s\in  [i_k+s_u+1,n-1]$, let $A_s$ be the submatrix of $M_k$ corresponding to $\{\varepsilon_{i_k}-\varepsilon_s, \varepsilon_s-\varepsilon_n, \varepsilon_s+\varepsilon_n  \} \subseteq Q_k \cup Y_{k+1}.$
For $t\in [i_k+1,i_k+s_u]$, let $A_t$ be the submatrix of $M_k$ corresponding to $\{ \varepsilon_t-\varepsilon_n, \varepsilon_t+\varepsilon_n  \} \subseteq Q_k \cup Y_{k+1}.$ For each $s,t \in [i_k+1, n-1]$, let $E_{s,t}=\text{antidiag}(1,1)$.
It is easy to see that the matrix $M_k$ has the form as given in Proposition 
\ref{p_det_D}(2), thus $\det M_k \equiv 0$.

\medskip

\noindent (b) If $s_k=2$, we set $a=i_{k-1}+1$ and $b=i_k=a+1$. When $d \in \Delta_{\text{spec}}$,
the set $Q_k$ is $\{\varepsilon_{a}-\varepsilon_j, \varepsilon_{b}-\varepsilon_j\mid j\in [i_k+s_v+1,n-1]\}$ for some integer $v$. Given $\lambda \in \operatorname{Hom}_{\mathbb Z}( \mathfrak g_{\mathbb Z,d}(2),\mathbb Z)$ such that  $$\lambda(\varepsilon_a-\varepsilon_n)=\lambda(\varepsilon_b+\varepsilon_n)=1
~\text{and}~ \lambda(\varepsilon_j+\varepsilon_{j+1})=1 ~\text{for}~ j\in [i_k+1, i_k+s_v].$$
Then it is not difficult to check that $ \det M_k=\pm 1$ for this $\lambda$.

When $d \notin \Delta_{\text{spec}}$, the set $Q_k$ is $$\{\varepsilon_{a}-\varepsilon_j, \varepsilon_{b}-\varepsilon_l\mid j\in [i_k+s_c+1,n-1], l\in [i_k+s_d+1,n-1]\}$$ for some integers $c$ and $d$ which satisfy $2\leq s_c< s_d$. When $s\in [i_k+s_c+1, i_k+s_d]$, we consider the submatrix
$A_s$ of $M_k$ corresponding to
$$\{\varepsilon_{a}-\varepsilon_s,\  \varepsilon_s-\varepsilon_n, \ \varepsilon_s+\varepsilon_n  \} \subseteq Q_k \cup Y_{k+1}.$$
When $t\in [i_k+s_d+1, n-1]$, we consider the submatrix
$A_t$ of $M_k$ corresponding to
$$\{\varepsilon_{a}-\varepsilon_t, \ \varepsilon_{b}-\varepsilon_t,\  \varepsilon_t-\varepsilon_n, \ \varepsilon_t+\varepsilon_n  \} \subseteq Q_k \cup Y_{k+1}.$$
For each $s,t \in [i_k+s_c+1, n-1]$, set $E_{s,t}=\text{antidiag}(1,1)$.
It is easy to check that the matrix $M_k$ satisfies the conditions in Proposition 
\ref{p_det_D}(2), thus $\det M_k =0$ when $\text{char}~\mathbf k =2 $.
\end{proof}

\section{The second Case  of type $D$}

\setcounter{theorem}{1}
\subsection{Odd weights of the second case}
The set $\Delta^o_2$ can be identified with the set of sequences $(i_1, i_2,\cdots, i_k)$ of integers, satisfying the following conditions (1)-(4):
\begin{enumerate}
\item $1 \leq i_1<i_2< \cdots <i_k \leq n-2$,
\item $s_l \leq s_{l+2}$ for $l \in [1,k-2]$,
\item $s_l \equiv 0$ for $l \in [1,k]$ with $l \equiv k$,
\item $s_{k-1} \leq 2(n-i_k)$.
\end{enumerate}
where $s_1 =i_1$ and $s_l = i_l-i_{l-1}$ for $l \in [2,k]$,
through identifying $d \in \Delta^o$ with $(i_1,i_2,\cdots,i_k)$ where $\{i_l| l \in [1,k]\} = \{ i \in [1,n]| d(\alpha_i) =1\}$.

\medskip

Let $d = (i_1,i_2,\cdots,i_k) \in \Delta^o_2$.
As before, we set $s_1 =i_1$ and $s_l = i_l-i_{l-1}$ for $l \in [2,k]$. By the condition $(\spadesuit_D)$ in Section \ref{WDD}, it is not difficult to see that  $d$ is special if and only if it satisfies  following conditions:

\noindent (a) $s_1$ is even;

\noindent (b) For any two integers $u\leq v\in [1,k]$ and $u\equiv v \equiv k$, if $$s_{u-2}<s_u=s_{u+2}=\dots=s_{v-2}=s_v<s_{v+2},$$ then $s_{v+1}-s_{u-1}$ is even.

\begin{lemma}
Let $d = (r) \in \Delta^o$, then $r$ is even and $d \in \Delta_{\text{spec}}$. There exists a non-degenerate $\lambda \in \operatorname{Hom}_{\mathbb Z}( \mathfrak g_{\mathbb Z,d}(2),\mathbb Z)$.

\end{lemma}

\begin{proof}
For $\lambda \in \operatorname{Hom}_{\mathbb Z}(\mathfrak g_{\mathbb Z,d}(2), \mathbb Z)$, we denote
$$A_{\lambda,l}=(\lambda([e_{\varepsilon_i-\varepsilon_l},e_{\varepsilon_j+\varepsilon_l}])_{(i,j)\in [1,r]\times[1,r]}$$
for ${l\in [r+1,n]}$.
Then
\begin{eqnarray*}
\mathscr G_{d,\lambda} =
\begin{pmatrix}
& A \\
-{^tA} &
\end{pmatrix},
\end{eqnarray*}
where $A = \operatorname{Diag}(A_{\lambda,r+1}, \cdots, A_{\lambda,n})$.

Now we set $\lambda(e_{\varepsilon_i+\varepsilon_{i+1}})=1$ for $i\in [1,r]$, $i\equiv 1$ and $\lambda(e_\alpha)=0$ for other roots.  Since $r$ is even, it is easy to check that
$\det A_{\lambda,l} \in \{ \pm 1\}$ for $l\in [r+1,n]$
which implies $\det \mathscr G_{d,\lambda} \in \{ \pm 1\}$ for this $\lambda$.

\end{proof}

\begin{lemma}\label{Dcase2(1)}
Let $d = (i_1,i_2, \cdots, i_k) \in \Delta^o \cap \Delta_{\text{spec}}$. Then there exists a non-degenerate $\lambda \in \operatorname{Hom}_{\mathbb Z}( \mathfrak g_{\mathbb Z,d}(2),\mathbb Z)$.
\end{lemma}

\begin{proof}
For $k\geq 2$, as the same discussion in Section \ref{ssfai}, we can assume that $\lambda$ is faithful and, construct the sets $Q_l, P_l$ for $l\in [1,k-2]$. We set
$$M_{l+1} = (\lambda([e_{\alpha},e_{\beta}]))_{\alpha \in Q_l,\beta \in P_l},~~\text{for $l\in [1,k-2]$}$$
and $M_k = (\lambda([e_{\alpha},e_{\beta}]))_{\alpha, \beta \in Q_{k-1} \cup Y_{d,k}}$. Then
$$\det \mathscr G_{d,\lambda} = \pm (\prod_{l=2}^{k-1} \det M_l)^2 \det M_k.$$
Without loss of generality, we assume that $k$ is even, the discussion for $k$ odd is the same.

To compute $ \det M_k$, we need to consider the set $Q_{k-1}$ which has the following form
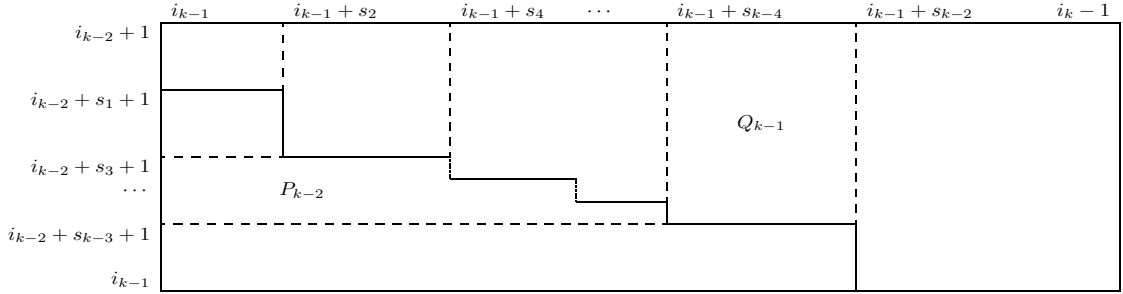
\begin{center}
\resizebox{\textwidth}{20mm}{
\begin{tabular}{r cccc cccc ccccc cccc ccccc}
    & $i_{k-1}$ & & & &
    $i_{k-1}+s_2$  &  & & &
    $i_{k-1}+s_4$ &&  $\cdots$ &&&
    $i_{k-1}+s_{k-4}$ & &&&
    $i_{k-1}+s_{k-2}$ &&&& $i_{k}-1$ \\
    \cline{2-23}
$i_{k-2}+1$ & \multicolumn{4}{|c;{4pt/4pt}}{\multirow{3}{*}{   }} &  \multicolumn{4}{c;{4pt/4pt}}{\multirow{3}{*}{  }} &
\multicolumn{5}{c;{4pt/4pt}}{\multirow{8}{*}{  }} &
\multicolumn{4}{c;{4pt/4pt}}{\multirow{9}{*}{ $Q_{k-1}$   }} &
\multicolumn{5}{c|}{\multirow{12}{*}{  }}
\\
& \multicolumn{4}{|c;{4pt/4pt}}{\multirow{3}{*}{ }} &
\multicolumn{4}{c;{4pt/4pt}}{\multirow{3}{*}{  }} &
\multicolumn{5}{c;{4pt/4pt}}{\multirow{8}{*}{  }} &
\multicolumn{4}{c;{4pt/4pt}}{\multirow{9}{*}{  }} &
\multicolumn{5}{c|}{\multirow{12}{*}{  }}
\\
& \multicolumn{4}{|c;{4pt/4pt}}{\multirow{3}{*}{ }} &
\multicolumn{4}{c;{4pt/4pt}}{\multirow{3}{*}{  }} &
\multicolumn{5}{c;{4pt/4pt}}{\multirow{8}{*}{  }} &
\multicolumn{4}{c;{4pt/4pt}}{\multirow{9}{*}{  }} &
\multicolumn{5}{c|}{\multirow{12}{*}{  }}
\\

\cline{2-5}
$i_{k-2}+s_{1}+1$ & \multicolumn{4}{|c|}{\multirow{3}{*}{ }} &
\multicolumn{4}{c;{4pt/4pt}}{\multirow{3}{*}{ }} &
\multicolumn{5}{c;{4pt/4pt}}{\multirow{8}{*}{ }} &
\multicolumn{4}{c;{4pt/4pt}}{\multirow{9}{*}{ }} &
\multicolumn{5}{c|}{\multirow{12}{*}{ }}
\\
& \multicolumn{4}{|c|}{\multirow{3}{*}{ }} &
\multicolumn{4}{c;{4pt/4pt}}{\multirow{3}{*}{ }} &
\multicolumn{5}{c;{4pt/4pt}}{\multirow{8}{*}{ }} &
\multicolumn{4}{c;{4pt/4pt}}{\multirow{9}{*}{ }} &
\multicolumn{5}{c|}{\multirow{12}{*}{ }}
\\
& \multicolumn{4}{|c|}{\multirow{3}{*}{ }} &
\multicolumn{4}{c;{4pt/4pt}}{\multirow{3}{*}{ }} &
\multicolumn{5}{c;{4pt/4pt}}{\multirow{8}{*}{ }} &
\multicolumn{4}{c;{4pt/4pt}}{\multirow{9}{*}{ }} &
\multicolumn{5}{c|}{\multirow{12}{*}{ }}
\\

\cdashline{2-5} \cline{6-9}
$i_{k-2}+s_3+1$ & \multicolumn{8}{|c;{1pt/1pt}}{\multirow{3}{*}{$P_{k-2}$ }} &
\multicolumn{5}{c;{4pt/4pt}}{\multirow{8}{*}{ }} &
\multicolumn{4}{c;{4pt/4pt}}{\multirow{9}{*}{ }} &
\multicolumn{5}{c|}{\multirow{12}{*}{ }}
\\
\cline{10-11}
$\cdots$  & \multicolumn{10}{|c;{1pt/1pt}}{\multirow{3}{*}{ }} &
\multicolumn{3}{c;{4pt/4pt}}{\multirow{8}{*}{ }} &
\multicolumn{4}{c;{4pt/4pt}}{\multirow{9}{*}{ }} &
\multicolumn{5}{c|}{\multirow{12}{*}{ }}
\\
\cline{12-14}
 & \multicolumn{13}{|c|}{\multirow{3}{*}{ }} &
\multicolumn{4}{c;{4pt/4pt}}{\multirow{9}{*}{ }} &
\multicolumn{5}{c|}{\multirow{12}{*}{ }}
\\
\cdashline{2-14} \cline{15-18}
$i_{k-2}+s_{k-3}+1$ & \multicolumn{17}{|c|}{\multirow{3}{*}{ }} &
\multicolumn{5}{c|}{\multirow{12}{*}{ }}
\\
& \multicolumn{17}{|c|}{\multirow{3}{*}{ }} &
\multicolumn{5}{c|}{\multirow{12}{*}{ }}
\\
$i_{k-1}$ & \multicolumn{17}{|c|}{\multirow{3}{*}{ }} &
\multicolumn{5}{c|}{\multirow{12}{*}{ }}
\\
\cline{2-23}
\end{tabular}
}
\end{center}

\smallskip

For $l\in [1,k-2]$, put
$$X_{l}=\{\varepsilon_s-\varepsilon_t\in \Phi_{d,2}\mid s\leq i_{l}<i_{l+1}<t~\text{and}~ t-s=i_{l+1}-i_{l-1}\},$$
and then put
$$X_{k-1}=\{\varepsilon_j+\varepsilon_{j+1} \in \Phi_{d,2} \mid j\in [i_{k-1}+1,i_{k}-1]~\text{and}~ j\equiv i_{k-1}+1\}.$$

Since $d\in \Delta^o$ is special, if there exist two even integers $u\leq v$ in $[2,k-2]$ satisfying
\begin{eqnarray} \label{e_con22}
s_{u-2}<s_u=s_{u+2}=\dots=s_{v-2}=s_v<s_{v+2},
\end{eqnarray}
 then $s_{v+1}-s_{u-1}$ is even.
For a pair of even integers $u \leq v$ in $[2,k-2]$ satisfying (\ref{e_con22}), we denote
$$\Omega_{u,v}=[i_{k-2}+s_{u-1}+1,\ i_{k-2}+s_{v+1} ].$$
Then we set
$$\Omega=\bigcup_{u \leq v\text{ in }[2,k-2] \text{ with } (\ref{e_con22})} \Omega_{u,v}.$$
The cardinality $\omega=|\Omega|$ is even. We write $\Omega$ as
$$\Omega=\{j_1,j_2,\dots, j_\omega\},$$
and set
$$X_{k}= \{\varepsilon_{j_s}+(-1)^s \varepsilon_{i_k+\lfloor\frac{s+1}{2}\rfloor}\mid s\in [1,\omega]\}.$$
With these settings, we denote $\mathbb{X}= \cup_{l = 1}^k X_{l}$ which is a subset of $\Phi_{d,2}$.

Let $\lambda: \mathfrak g_{\mathbb Z,d}(2) \rightarrow \mathbb Z$ be the homomorphism such that
$$\lambda(e_{\alpha}) =
\begin{cases}
1 & \text{ if } \alpha \in \mathbb{X}, \\
0 & \text{ if } \alpha \in \Phi_{d,2} \backslash \mathbb{X}.
\end{cases}$$
In the following we show that $\det \mathscr G_{d,\lambda} \in \{ \pm 1\}$.

Obviously, $\lambda$ is faithful in the sense of \ref{ssfai}, then $$\det \mathscr G_{d,\lambda} = \pm (\prod_{l=2}^{k-1} \det M_l)^2 \det M_k.$$
It is easy to check that
$\det M_l \in \{ \pm 1\}$ for $l \in [2,k-1]$.

To compute $\det M_k$, we put
$$x_r=\lambda(\varepsilon_r+\varepsilon_{r+1})~\text{and}~y_s= \lambda(\varepsilon_{j_s}+(-1)^s \varepsilon_{i_k+\lfloor\frac{s+1}{2}\rfloor})$$
for $r\in [i_{k-1}+1,i_{k}-1]$ with $r\equiv i_{k-1}+1$, and $s\in [1,\omega]$.

Note that
$$ M_k = (\lambda([e_{\alpha},e_{\beta}]))_{\alpha, \beta \in Q_{k-1} \cup Y_{d,k}}.$$
So we divide $Q_{k-1} \cup Y_{d,k}$ into subsets and consider the submatrices of $\mathscr G_{d,\lambda}$ corresponding to these subsets.

\medskip

\noindent (a) One type of subsets are of the form
$$\{\varepsilon_{j_s}-\varepsilon_l,\  \varepsilon_{l}+(-1)^s \varepsilon_{i_k+\lfloor\frac{s+1}{2}\rfloor},
\ \varepsilon_{l+1}+(-1)^{s+1} \varepsilon_{i_k+\lfloor\frac{s}{2}\rfloor+1},
\ \varepsilon_{j_{s+1}}-\varepsilon_{l+1}\},$$
or
$$\{\varepsilon_{j_s}-\varepsilon_{l+1},\  \varepsilon_{l+1}+(-1)^s \varepsilon_{i_k+\lfloor\frac{s+1}{2}\rfloor},
\ \varepsilon_{l}+(-1)^{s+1} \varepsilon_{i_k+\lfloor\frac{s}{2}\rfloor+1},
\ \varepsilon_{j_{s+1}}-\varepsilon_{l}\},$$
where $s\in [1,\omega]$ and $l\in [i_{k-1}+1,i_{k}-1]$ satisfy $s\equiv 1$, $l\equiv i_{k-1}+1$ and $\varepsilon_{j_\sigma}-\varepsilon_\tau \in Q_{k-1}$ for $\sigma\in [1,\omega]$, $\tau\in [i_{k-1}+1,i_{k}-1]$.

\medskip
\noindent (b) The other type of subsets are of the form
$$\{\varepsilon_j-\varepsilon_l,\ \varepsilon_{j+1}+\varepsilon_l,\  \varepsilon_j+\varepsilon_l, \ \varepsilon_{j+1}-\varepsilon_l \},$$
and these subsets cover all the remaining roots in $Q_{k-1} \cup Y_{d,k}$.

Let $X$ be a subset of $Q_{k-1} \cup Y_{d,k}$ of type (a) or (b). Then
$$(\lambda([e_{\alpha},e_{\beta}]))_{\alpha,\beta \in X} =
\begin{cases}
\begin{pmatrix}
&0   & y_s   & 0  & 0  \\
&-y_s  & 0   & x_l   & 0   \\
&0   & -x_l    & 0  & y_{s+1}  \\
&0   & 0    & -y_{s+1}  & 0
\end{pmatrix} & \text{ if } X \text{ is of type (a)}, \\
\begin{pmatrix}
&0   &x_j  &0   &0 \\
&-x_j  &0  &0   &0 \\
&0   &0   &0   &x_j \\
 &0   &0 &-x_j  &0
\end{pmatrix} & \text{ if } X \text{ is of type (b)}.
\end{cases}$$

It is easy to see that $M_k$ can be written as a blocked diagonal matrix with diagonal blocks of the form $(\lambda([e_{\alpha},e_{\beta}]))_{\alpha, \beta \in X}$ for each subset $X$ of type (a) or (b), and each block is of determinant $\pm 1$. Thus $\det M_k \in \{ \pm 1\}$, which implies $\det \mathscr G_{d,\lambda} \in \{ \pm 1\}$.

\end{proof}

\begin{lemma} \label{Dcase2(2)}
Let $d = (i_1,i_2, \cdots, i_k) \in \Delta^o \backslash \Delta_{\text{spec}}$.
If $\operatorname{char}\mathbf k =2$, then $\det \mathscr G_{d,\lambda} =0$ for any $\lambda \in \operatorname{Hom}(\mathfrak g_{d}(2), \mathbf k)$.
\end{lemma}

\begin{proof}
Suppose that there exists $\lambda \in \operatorname{Hom}(\mathfrak g_{d}(2), \mathbf k)$  such that $\det \mathscr G_{d,\lambda} \ne 0$. We can assume that $\lambda$ is faithful by 
(\ref{e_iso_eq}) and Lemma \ref{l_iso_faith}.
So by (\ref{e_prod}), $\det \mathscr G_{d,\lambda} \ne 0$ implies that $\det M_k \ne 0$, where $M_k$ is defined as in (\ref{e_MK}). In the following we assume that $k$ is even without loss of generality.

For each $l\in [i_{k-1}+1,i_k]$, there exist two even integers $u < v$ in $[0,k-2]$ satisfying
\begin{eqnarray}
s_{u-2}<s_u=s_{u+2}=\dots=s_{v-2}=s_v<s_{v+2}, \nonumber
\end{eqnarray}
and $i_{k-1}+s_u< l\leq i_{k-1}+s_v$. We denote $\Xi_l=[i_{k-2}+1, i_{k-2}+s_{v+1}]$.
Since $d$ is not special, there exists an integer $l\in [i_{k-1}+1,i_k]$ such that $|\Xi_l|$ is odd.

For each fixed $s\in [i_{k-1}+1,i_k]$, we set $$A_{s}=\lambda([e_{\varepsilon_r-\varepsilon_s},  e_\varpi ])$$
where $r\in \Xi_s $ and $\varpi$ runs over the set
$\{\varepsilon_s-\varepsilon_l,  \ \varepsilon_s+\varepsilon_l \mid   l\in [i_k+1,n]\}.$
For any $s,t\in [i_{k-1}+1,i_k]$, denote $E_{s,t}$ the block associate the set of roots
$$\{\varepsilon_s-\varepsilon_l,\  \varepsilon_s+\varepsilon_l, \ \varepsilon_t-\varepsilon_l, \ \varepsilon_t+\varepsilon_l\mid l\in [i_k+1,n] \}.$$

With a good ordering in $Q_k\cup Y_{d,k}$, the matrix $M_k$ has the same form as the matrix $S$ in Proposition 
\ref{p_det_D}(2). Since $d$ is not special, there exists $s\in [i_{k-1}+1,i_k]$ such that $|\Xi_s|$ is odd.
 Then all conditions in Proposition 
 \ref{p_det_D}(2) are satisfied and  we have $\det M_k =0$ when $\text{char}~\mathbf k =2 $. We get a contradiction which proves the lemma.

\end{proof}

Combining Lemma \ref{Dcase1}, Lemma \ref{Dcase2(1)} and Lemma \ref{Dcase2(2)}, we have proved Theorem \ref{mainThm2} in the type $D$ case. Therefore with  the results we get in Sections 3 to 8, we have proved
Theorem \ref{mainThm2}.

\section{Final remarks}

Theorem \ref{mainThm1}, combined with \cite[Corollary 5.11]{G}, gives a new characterization of special unipotent classes.
According to the discussion in Section 0 and Theorem \ref{mainThm1}, the GGGRs can be defined for certain unipotent orbits coming from characteristic zero.

In the good characteristic case, Kawanaka's definition of GGGRs starts with a unipotent element. This element will determine a weighted Dynkin diagram and subsequently a grading on the Lie algebra. This grading is crucial for the definition. Maybe that's why Geck's definition starts with a weighted Dynkin diagram, from which the required grading can be obtained. But in the bad characteristic case, not the set of unipotent orbits, but the set of unipotent orbits coming from characteristic zero, is parameterized by the set of weighted Dynkin diagram.
Thus, there is no "Kawanaka's conjecture" in this case, until we are able to define a GGGR for each unipotent orbit. For Kawanaka's conjecture in good characteristic case, see \cite[Conjecture 3.3.1]{K2}, \cite[Theorem 2.4.3]{K3} and \cite[Theorem 4.5]{GH}.

Assuming $p,q$ to be large enough, Lusztig gave the block decomposition of the characters of GGGRs, and expressed each block component in terms of characteristic functions of intersection cohomology complexes in this block, see \cite{L1}. The assumption of Lusztig was reduced by Taylor in \cite{T}.
According to \cite{X}, there is only one block of intersection cohomology complexes, when $G$ is of type $B_n$, $C_n$ or $D_n$ and $p=2$. So there is no need for a block decomposition. But it may still be interesting to express the characters of GGGRs in terms of characteristic functions of intersection cohomology complexes.

$\\$

\begin{flushleft}
\textbf{Acknowledgement.} The authors are grateful to Professor Toshiaki Shoji for his
encouragement and helpful suggestions. The second named author would like to thank Professor Meinolf Geck for his detailed and enlightening answers to many questions, and thank Professor Leonard Scott and Professor Daniel Nakano for useful suggestions during ICRT 8.
\end{flushleft}


\begin{thebibliography}{10}
\bibitem[C]{C} R. W. Carter, Finite groups of Lie type: Conjugacy classes and complex characters, Wiley, New York, 1985.
\bibitem[E]{E} H. Enomoto, The characters of the finite symplectic group $Sp(4,q)$, $q =2^f$, saka J. Math. 9 (1972), 75-94.
\bibitem[G]{G} M. Geck,  Generalised Gelfand--Graev representations in bad characteristic?, arXiv:1810.08937 (2018).
\bibitem[GH]{GH} M. Geck and H\'{e}zard, On the unipotent support of character sheaves, Osaka J. Math. 45 (2008), 819-831.
\bibitem[GM]{GM} M. Geck and G. Malle, On the existence of a unipotent support for the irreducible characters of finite groups of Lie type, Trans. Amer. Math. Soc. 352 (2000), 429-456.
\bibitem[H]{H} W. H. Hessenlink, Nilpotency in classical groups over a field of characteristic 2, Mathematische Zeitschrift 166.2 (1979): 165-181.
\bibitem[K1]{K0} N. Kawanaka, Fourier transforms of nilpotently supported invariant functions on a simple Lie algebra over a finite field, Invent. Math., 69 (1982), 411-435.
\bibitem[K2]{K1} N. Kawanaka, Generalized Gelfand-Graev representations and Ennola duality, in: Algebraic Groups and Related Topics, Advanced Studies in Pure Math. 6, Kinokuniya, Tokyo, and North-Holland, Amsterdam, 1985, 175-206.
\bibitem[K3]{K2} N. Kawanaka, Generalized Gelfand-Graev representations of exceptional algebraic groups I, Invent. Math. 84 (1986), 575-616.
\bibitem[K4]{K3} N. Kawanaka, Shintani lifting and Gelfand-Graev representations, Proc. Symp. Pure Math., Amer. Math. Soc. 47 (1987), 147-163.
\bibitem[L]{L1} G. Lusztig, A unipotent support for irreducible representations, Adv. Math. 94 (1992), 139-179.
\bibitem[Sp]{S}  N. Spaltenstein, Classes unipotentes et sous-groupes de Borel, Lecture Notes in Math., vol. 946, Springer, Berlin Heidelberg New York, 1982.
\bibitem[St]{St} R. Steinberg, Lectures on Chevalley groups, Yale University 1967.
\bibitem[T]{T} J. Taylor, Generalized Gelfand-Graev representations in smalll characteristics, Nagoya Math. J. 224 (2016), 93-167.
\bibitem[X]{X} T. Xue,  Nilpotent orbits in bad characteristic and the Springer correspondence. Diss. Massachusetts Institute of Technology, 2010.

\end{thebibliography}
\end{document}